\documentclass[12pt]{article}

\usepackage[a4paper, left=2.9cm, right=2.9cm, top=3.cm, bottom=3.cm]{geometry}
\usepackage[T1]{fontenc}
\usepackage[utf8]{inputenc}

\usepackage[T1]{fontenc}
\usepackage[helvratio=0.92]{newtxtext}

\usepackage[title]{appendix}

\usepackage{graphicx}
\usepackage{xfrac}
\usepackage{newtxtext}
\usepackage{amsmath}
\usepackage{amsthm}
\usepackage{amsfonts}
\usepackage{amssymb}
\usepackage{graphicx}
\usepackage{relsize}
\numberwithin{equation}{section}
\usepackage{color}
\usepackage{caption}
\usepackage{pdfrender,xcolor}
\usepackage{enumitem}
\usepackage{amssymb,amsmath,amsthm,enumitem}
\usepackage{enumitem}
\usepackage[arrow, matrix, curve]{xy}
\usepackage{xcolor}
\usepackage{bigints}

\usepackage[bookmarks=false]{hyperref}

\makeatletter
\renewcommand{\fnum@figure}{Fig. \thefigure}
\makeatother

\usepackage[T1]{fontenc}
\usepackage[utf8]{inputenc}

\DeclareRobustCommand{\rchi}{{\mathpalette\irchi\relax}}
\newcommand{\irchi}[2]{\raisebox{\depth}{$#1\chi$}} 

\usepackage{setspace}
\usepackage{epstopdf}
\usepackage[caption=false]{subfig}

\usepackage{color}
\DeclareMathOperator*{\esssup}{ess\,sup}
\makeatletter
\def\NAT@def@citea{\def\@citea{\NAT@separator}}
\makeatother

\usepackage[T1]{fontenc}
\usepackage[helvratio=0.92]{newtxtext}
\usepackage{newtxmath}
\usepackage{microtype}
\usepackage{authblk}
\usepackage{fancyhdr}

\emergencystretch=\maxdimen
\theoremstyle{plain}
\newtheorem{theorem}{Theorem}[section]
\newtheorem{lemma}[theorem]{Lemma}

\newtheorem{proposition}[theorem]{Proposition}

\theoremstyle{definition}
\newtheorem{definition}[theorem]{Definition}

\newtheorem{remark}[theorem]{Remark}

\newcommand{\R}{\mathbb{R}}
\newcommand{\Z}{\mathbb{Z}}
\newcommand{\N}{\mathbb{N}}
\newcommand{\Oe}{\Omega_\epsilon}

\usepackage{amsthm}
\usepackage{xpatch}
\xpatchcmd{\proof}{\itshape}{\normalfont\proofnamefont}{}{}
\newcommand{\proofnamefont}{}
\renewcommand{\proofnamefont}{\bfseries}

\usepackage{setspace}

\allowdisplaybreaks
\usepackage{titling}

\usepackage{blindtext}

\newcommand\shorttitle{Homogenization of Poisson--Nernst--Planck equations}
\newcommand\authors{A. Bhattacharya}

\usepackage{amsmath,amssymb,amsfonts,amscd}

\begin{document}

	\title{\Large Homogenization of Poisson--Nernst--Planck equations for multiple species in a porous medium}
	
	\author{Apratim Bhattacharya}

	\date{\small \vspace{.7cm}Department of Mathematics and Mathematical Statistics, Umeå University,\\
 MIT-building, floor 3, 901 87 Umeå, Sweden.\\ 
 Email: apr.bhattacharya@gmail.com}
	\maketitle
	
	\begin{abstract}
We rigorously derive a homogenized model for the Poisson--Nernst--Planck (PNP) equations for the case of multiple species defined on a periodic porous medium in spatial dimensions two and three. This extends the previous homogenization results for the PNP equations concerning two species. Here, the main difficulty is that the microscopic concentrations remain uniformly bounded in a space with relatively weak regularity. Therefore, the standard Aubin-Lions-Simon type compactness results for porous media, which give strong convergence of the microscopic solutions, become inapplicable in our weak setting. We overcome this problem by constructing suitable cut-off functions. The cut-off function, together with the application of a previously known energy functional, yields strong convergence of the microscopic concentrations in $L^1_t L^r_x$, for some $r>2$, enabling us to pass to the limit in the nonlinear drift term. Finally, we derive the homogenized equations by means of two-scale convergence in $L^p_t L^q_x$ setting. 
	\end{abstract}
	
	

	\section{Introduction}
	
The Poisson--Nernst--Planck (PNP) equations form a coupled parabolic-elliptic system, which models the transport of charged particles under the influence of diffusion and electric force. The goal of this paper is the rigorous homogenization of the PNP equations for multiple species in a periodic porous medium in spatial dimensions two and three. 

\subsection{Motivation} The PNP system is extensively applied to model electro-diffusion phenomena \cite{Soe10, Baz04, Wan14, War20}. Furthermore, PNP type equations are widely used in the mathematical modeling of semiconductors (see \cite[Chapter 3]{Mar90}, \cite[Chapter 2]{Sel84}, \cite{Gaj86}, \cite{Doa19}), where the case of two species is relevant (the electrons with charge number $-1$ and the holes with charge number $1$). There are also mathematical studies where the PNP equations were generalized suitably in order to incorporate specific complex physics (see \cite{Bur12,Rou07}).

In this paper, we are interested in the applications of the PNP system in a medium with microscopic heterogeneities. In biology, examples of such phenomena include ion transportation in biological tissues, such as ion transfer through ion channels of the cell membrane \cite[Chapter 3]{Kee09} and application to neuronal signal propagation \cite{Pod13} as well as understanding disease characteristics \cite{Bor12}. In geology, these equations on complex media appear when modeling electro-kinetic flow through porous rocks \cite{Ali19} and, in engineering sciences, when modeling electro-osmosis in porous media \cite{Wan07}. 

In practice, when dealing with such processes with microscopic complexities, one often seeks a global or effective behaviour. This is where homogenization theory is utilized. More precisely, in this paper, we perform a rigorous homogenization of (\ref{non_dim_PNP}) on the setting of a periodic porous domain, and this homogenized model provides a macroscopic approximation of the original microscopic process. Furthermore, a homogenized model is defined on a domain which is free of micro-structures. Hence, it is far more suitable for numerical computations. Also, the problem (\ref{non_dim_PNP}) is intriguing from the point of view of multiscale analysis and mathematical homogenization due to the nonlinear couplings via the drift terms. 

\subsection{The mathematical model} We study the following system of equations for a number of $P \in \N$ species with concentrations $c_{i, \epsilon}, 1 \leq i \leq P$, and the electric potential $\phi_\epsilon$ in a periodically perforated domain $\Omega_\epsilon$, representing the fluid (pore) phase of a porous medium (see Fig. \ref{fig_domain_PNP}):

	\begin{subequations}\label{non_dim_PNP}
		\begin{align}
			\partial_{t} c_{i,\epsilon}(t,x) +\nabla \cdot J_{i,\epsilon} (t,x) &=0  && \mbox{ in }  (0,T) \times \Omega_\epsilon, \label{non_dim_NP_eq}
			\\
			J_{i,\epsilon} (t,x) \cdot \nu_\epsilon &=0   && \mbox{ on }  (0,T) \times \partial \Omega_\epsilon, \label{non_dim_bc_1}
			\\
			c_{i,\epsilon}(0,x) &= c^{0}_{i} (x)  && \mbox{ in } \Oe, \label{non_dim_ic}
			\\
			-   \Delta \phi_\epsilon (t,x) &=  \sum_{i=1}^{P}  z_i c_{i,\epsilon}  (t,x)   && \mbox{ in }  (0,T) \times \Omega_\epsilon, \label{non_dim_poisson_eq}
			\\
			\nabla \phi_\epsilon (t,x) \cdot \nu_\epsilon &  = \xi_\epsilon (x)    && \mbox{ on } (0,T) \times \partial \Omega_\epsilon \label{non_dim_bc_2},
		\end{align}
	\end{subequations}
	where the total flux $J_{i,\epsilon}$ of the $i$-th species is given by
	\begin{equation*}\label{total_flux}
		J_{i,\epsilon}(t, x) = - \left(  D_{i} (t,x) \nabla c_{i,\epsilon} (t,x)+   D_{i} (t,x) z_i c_{i,\epsilon} (t,x) \nabla \phi_\epsilon (t,x) \right).
	\end{equation*}
	Here $(0,T)$ is the time interval, whereas $z_i \in \Z$ and $D_i $ denote the charge number and diffusivity of the $i$-th species, respectively. The microscopic scale parameter $\epsilon >0$ describes the length of the period of the porous microstructure and it is proportional to the radius of the pores. $\nu_\epsilon$ represents the outward unit normal vector to the boundary $\partial \Omega_\epsilon$. 
		\begin{figure}%
		\centering
		{{\includegraphics[width=6.5cm]{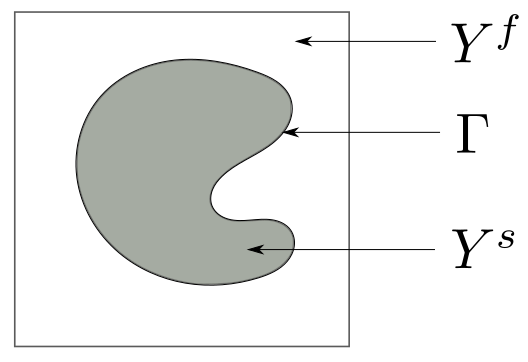} }}%
		\qquad 
		{{\includegraphics[width=7cm]{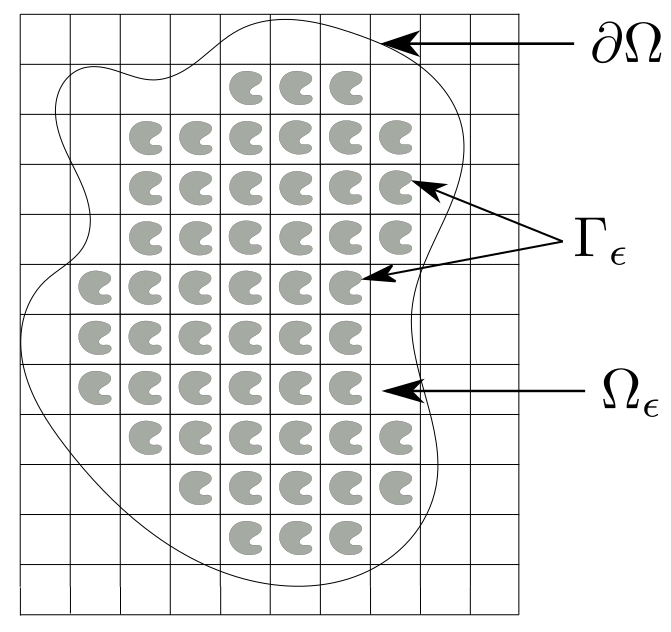} }}%
		\caption{The standard cell $Y= Y^f \cup \overline{Y^s}$ (left) and the porous medium $\Omega$ with the fluid (pore) part $\Omega_\epsilon$  (right). \label{fig_domain_PNP}}
	\end{figure}
	
Equations (\ref{non_dim_PNP}) model the transport of the charged species due to diffusion and electric force in a domain with an impermeable boundary, modeled by the no-flux boundary condition (\ref{non_dim_bc_1}). The initial concentrations are given in (\ref{non_dim_ic}). The electric potential is induced by the charges of the species and is given as the solution of the Poisson's equation (\ref{non_dim_poisson_eq}) subject to the non-homogeneous Neumann boundary condition (\ref{non_dim_bc_2}). The right-hand side in \eqref{non_dim_poisson_eq} represents the charge density within the pore phase of the porous medium, while $\xi_\epsilon$ in (\ref{non_dim_bc_2}) models a charged boundary. If we consider, e.g., the case of ion channels located within the cell membranes, it is known that there is a charge distribution on the boundary of the channel's pore. This boundary charge, which in our model is described by $\xi_\epsilon$, has a significant role in the properties of the channel, such as selectivity \cite{Lea97, Doy98}. Let us mention that consideration of other kinds of boundary conditions will also be interesting from the point of view of analysis and applications.

The aim of this paper is to study the asymptotic behaviour of the microscopic solutions $(c_{i,\epsilon}, \phi_{\epsilon})$ as $\epsilon \to 0$.

\subsection{Previous contributions}
Rigorous homogenization results have been obtained for the PNP system coupled to the Stokes equation for the case of two charged species with opposite charge numbers, e.g., in \cite{Sch11, Ray12}. Furthermore, in \cite{All10}, a stationary and linearized Stokes-PNP system for multiple species was homogenized. Corrector results related to the model from \cite{Ray12} can be found in \cite{Kho19}. Homogenization of an approximate version of the PNP system having relatively regular solutions is done in \cite{Bha22}. We denote this approximate system by app-PNP (see Subsection \ref{app-PNP_def} for its definition). Later we will see that this system plays a central role in our paper. On the other hand, in \cite{Kov21} a generalized PNP problem for multiple species in a two-phase medium, with transmission conditions at the microscopic interface, has been homogenized. However, the solutions in \cite{Kov21} enjoy the total mass balance property, which is not satisfied by the classical PNP system (\ref{non_dim_PNP}), making our analysis significantly more challenging. The rigorous homogenization of the classical PNP equations for the case of multiple species is not available in the literature so far, and, to our knowledge, our paper is the first one to present such a result. This extends the previous homogenization result for the classical PNP equations with two species. 

\subsection{Mathematical challenge}
It has been noted that the multiple-species case causes more difficulty in the analysis compared to the two-species case (for a discussion on this, see \cite[Introduction and main results]{Both14} and \cite[Introduction]{Jou23}). More precisely, when dealing with two species with opposite charge numbers, certain terms can be shown to have a sign, which then helps in obtaining estimates (see \cite[proof of Theorem 3.4]{Ray12}). Furthermore, generally speaking, the existence of solutions for the PNP system with multiple species has been obtained in weaker spaces compared to the two-species case (compare \cite{Gaj85, Her12}, which study the two-species case, to \cite{Both14, Wu15} dealing with the multiple-species case). Comments on the low regularity of solutions to our model can be found in Remark \ref{remark_low_reg}. 

In homogenization theory, one needs uniform estimates for the microscopic solutions $(c_{i,\epsilon}, \phi_\epsilon)$ with respect to the microscopic scale parameter $\epsilon$. These uniform estimates then help pass to the limit in the microscopic model (\ref{non_dim_PNP}) as $\epsilon \to 0$ and obtain the homogenized equations. Moreover, we need strong convergence of the concentrations $c_{i, \epsilon}$ in order to pass to the limit in the nonlinear drift term in (\ref{non_dim_NP_eq}). In this regard, let us mention that the application of the Aubin-Lions-Simon compactness lemma \cite{Sim87}, which gives strong convergence, to the setting of porous media is not straightforward. One of the major difficulties is the fact that one has no information on the time derivative of the extension of the solution (see \cite[p. 1830]{Gah16} for more details; see Lemma \ref{lemma_extension} for details on the extension operator). In \cite{Mei11}, \cite{Gah16}, strong convergence of the extensions of solutions has been obtained by means of fine analyses which suitably utilizes the Aubin-Lions-Simon compactness lemma \cite{Sim87} in the setting of porous media. However, lack of regularity of our solutions makes the arguments from \cite{Mei11, Gah16} inapplicable to our setting. In the following, let us discuss this issue in detail.

In \cite{Mei11} strong convergence of the extensions of solutions is proved in $L^2_tL^2_x$ under the assumption that the solutions are bounded in $L^2_tH^1_x$ and their time derivatives are bounded in $L^2_t\left(H^1_x\right)^\prime$. This result was extended in \cite{Hoe16}, where strong convergence was obtained in $L^p_tL^p_x$, provided that the solutions and their time derivatives are bounded in $L^{p}_t W^{1,p}_x$ and $L^{q}_t\left(W^{1,p}_x\right)^\prime$, respectively, where $2 \leq p <\infty$ and $p^{-1}+  q^{-1} =1$. Such assumptions in \cite{Mei11, Hoe16} allow application of the Lions–Magenes lemma, which says that if $V \subseteq H \subseteq V^\prime$ is an evolution triple, then the set of functions in $L^p_tV$ with time derivatives in $L^q_tV^\prime$ are continuously embedded in $C_tH$, where $1 < p <\infty$ and $p^{-1}+  q^{-1} =1$ (see, e.g., \cite[Proposition 23.23 (ii)]{Zei90}). Consequently, the solutions in \cite{Mei11, Hoe16} remain uniformly bounded in $C_tL^2_x$, which turns out to be crucial for the proofs (see \cite[Lemma 3.3]{Mei11}, \cite[Chapter 5, Lemma 5.5]{Hoe16}). However, in our situation the Lions-Magenes lemma is not applicable. Indeed, although we have uniform bounds for the microscopic concentrations $c_{i,\epsilon}$ in $L^{\frac{5}{4}}_tW^{1,\frac{5}{4}}_x$ (see Proposition \ref{uni_est_thm_2_st}), $\partial_t c_{i,\epsilon}$ is not uniformly bounded in $L^{5}_t\left(W^{1,\frac{5}{4}}_x\right)^\prime$. Next, let us show this point. From Proposition \ref{uni_est_thm_2_st} and applying the Sobolev embedding for $c_{i,\epsilon}$ we get 
\begin{equation}\label{difficulty1}
\lVert c_{i,\epsilon} \rVert_{L^{\frac{5}{4}}_t L^{\frac{15}{7}}_x} + \lVert \nabla c_{i,\epsilon} \rVert_{L^{\frac{5}{4}}_t L^{\frac{5}{4}}_x} + \lVert \nabla \phi_{\epsilon} \rVert_{L^{\infty}_t L^{2}_x} \leq C,
\end{equation}
for some $C>0$ independent of $\epsilon$. Now, using the boundedness of $D_i$ and H{\"o}lder's inequality, we estimate: 
\begin{eqnarray*}
&&\int_{\Omega_\epsilon} ( D_{i} \nabla  c_{i,\epsilon} + D_{i} z_i c_{i,\epsilon} \nabla \phi_\epsilon) \nabla \psi \,dx \\
&&\leq C\left( \lVert \nabla c_{i,\epsilon} (t) \rVert_{L^{\frac{5}{4}}_x} \lVert \nabla \psi  \rVert_{L^{5}_x} 
 +  \lVert  c_{i,\epsilon} (t) \rVert_{L^{\frac{15}{7}}_x}  \lVert \nabla \phi_{\epsilon} (t) \rVert_{L^{2}_x} \lVert \nabla \psi  \rVert_{L^{30}_x}\right). 
\end{eqnarray*}
Consequently, (\ref{difficulty1}) and the weak formulation (\ref{Exist_1_micro}) give that $\partial_t c_{i,\epsilon}$ is uniformly bounded in $L^{\frac{5}{4}}_t \left(W^{1,30}_x\right)^\prime$. Clearly, the above argument does not work if $\psi \in W^{1,5}_x$. Therefore, we have no control of $\partial_t c_{i,\epsilon}$ in $L^{5}_t \left(W^{1,5}_x\right)^\prime$, making the Lions-Magenes lemma inapplicable to our setting.

On the other hand, in \cite{Gah16} strong convergence of the extensions of solutions is proved in $L^2_tL^2_x$ by suitably controlling their shifts with respect to time (see \cite[Lemmas 9 and 10]{Gah16}). There, the underlying assumption is that the solutions and their time derivatives are bounded in $L^2_tH^1_x$ and $L^2_t\left(H ^1_x\right)^\prime$, respectively. However, the arguments in \cite{Gah16} rely heavily on such assumptions, and there seem to be no possibility of adapting them to our less regular setting.

Now, in the following, we discuss the main idea of deriving uniform estimates for our microscopic solutions as well as obtaining the required strong convergence.  

\subsection{Main idea}

First, let us mention the strategy of obtaining the uniform estimates for the microscopic solutions $(c_{i,\epsilon}, \phi_{\epsilon})$ with respect to $\epsilon$. In \cite{Both14}, an approximate PNP (app-PNP) system was introduced by replacing the linear diffusion in (\ref{non_dim_PNP}) by a nonlinear diffusion, $h^\eta_p(r)= r+ \eta r^p$ (see (\ref{non_dim_app_PNP}), (\ref{def_h})), in order to prove global existence of weak solutions to the PNP system. Specifically, the solutions of the app-PNP system converge to the solutions of the PNP system as $\eta \to 0$ (Lemma \ref{conv_approx_micro}). Furthermore, for any fixed $\eta$, the uniform estimates for the microscopic solutions of the app-PNP with respect to the microscopic scale parameter $\epsilon$ have been obtained in \cite{Bha22}. This has been done in order to derive a homogenized model for the app-PNP, for any fixed $\eta >0$, in \cite{Bha22}. Here, we extend the result of \cite{Bha22} and obtain the estimates for the microscopic solutions of the app-PNP uniformly with respect to both the approximation parameter $\eta$ and the microscopic scale parameter $\epsilon$. We obtain this by using the fact that there exists an energy functional associated with the app-PNP system \cite[p. 157]{Both14}, which is non-increasing in time along the solutions \cite[Proposition 3.1]{Bha22}. Finally, these uniform estimates for the microscopic app-PNP as well as the convergence of the app-PNP to the PNP enable us to get the required estimates for the microscopic PNP uniformly with respect to $\epsilon$ (see Theorem \ref{uni_est_thm_2_st}).

Next, we discuss the key contribution of this paper (Theorem \ref{thm_st_conv_micro_1}), which is obtaining strong convergence in the setting of porous media when the microscopic solutions are bounded in a relatively weak space. Indeed, we only obtain uniform bound for $c_{i,\epsilon}$ in $L^\frac{5}{4}_tW^{1,\frac{5}{4}}_x$ (see Proposition \ref{uni_est_thm_2_st}) as well as some energy bounds for the app-PNP system (see Subsection \ref{subsec_est_appPNP}). Again, the idea is to make use of the Aubin-Lions-Simon lemma \cite[Theorem 3]{Sim87} suitably to our setting. As in \cite{Sim87}, to obtain the strong convergence, we prove an equicontinuity type property for the microscopic concentrations. In order to do so, we introduce appropriate cut-off functions (see Subsection \ref{subsec_cut_off}). The cut-off function helps us obtain an estimate for the solution on the part of the domain where the solution is less than or equal to the height of the cut-off function. More precisely, on this part of the domain, the estimate is obtained in a space which is more regular; moreover, this relatively regular estimate depends only on the height of the cut-off function (see Proposition \ref{th_st_cut_off}). This helps us obtain equicontinuity when the values of the solutions are less than or equal to the height of the cut-off function (see estimate of I in Theorem \ref{thm_st_conv_micro_1}). Furthermore, it turns out that when the solutions are larger than the height of the cut-off function, the application of the energy functional is enough in order to obtain the required control on the solutions (see estimates of II, III and IV in Theorem \ref{thm_st_conv_micro_1}). These arguments eventually guarantee the strong convergence of the microscopic concentrations in $L^1_tL^r_x$, for some $r>2$, which is crucial to pass to the limit, as $\epsilon \to 0$, in the non-linear drift term in (\ref{non_dim_NP_eq}). 

Finally, we derive the homogenized model by means of two-scale convergence in $L^p_t L^q_x$ setting. Since the estimates for the solutions are obtained in a weaker space, this setting is suitable for our problem instead of the classical $L^2_tL^2_x$ setting. 

\begin{remark}\label{remark_low_reg}
Although, for each $\eta$ the solutions to the app-PNP system live in a regular space, such regular norms are not uniformly bounded with respect to $\eta$. Indeed, from (\ref{eta_blow_up_eq}) we see that the bound of the norm $\left\lVert c^\eta_{i,\epsilon}\right\rVert_{L^\infty_tL^p_x}$ blows up as $\eta \to 0$. Therefore, we do not have any information on whether $\left\lVert c_{i,\epsilon}\right\rVert_{L^\infty_tL^p_x}$ is finite in the limit. Due to this lack of information, we cannot use a fixed point argument from \cite[Proposition 2.1]{Bha22} and, consequently, cannot conclude that $c_{i,\epsilon} \in L^2_tH^1_x \cap L^\infty_t L^\infty_x$ and $\phi_\epsilon \in L^\infty_t W^{1,\infty}_x$.
\end{remark}

\begin{remark}
Since we are dealing with Poisson's equation with the Neumann boundary condition for the electric potential, note that if $\phi_\epsilon$ is a solution, so is $\phi_\epsilon + C$, for any constant $C$. For the electric potential satisfying the app-PNP, we fix a $C$ (which is needed to obtain the existence result \cite[Proposition 2.1]{Bha22}) by imposing the mean-zero condition. Consequently, the convergence of the app-PNP to PNP implies that $\phi_\epsilon$ also has mean-zero (see Remark \ref{mean_value_zero}).
\end{remark}

\begin{remark}
Throughout this paper, the results hold in both two and three spatial dimensions. More precisely, the embeddings are done for the three-dimensional case; the corresponding embeddings in the two-dimensional setting follow automatically. In fact, for the two-dimensional case, the embeddings hold in more regular spaces. However, for simplicity we do not distinguish between the two- and three-dimensional cases in this paper.
\end{remark}

 \subsection{Outline of the paper}
 This paper is organized as follows. In Section \ref{micro_prob}, the setting of the microscopic model is introduced. Specifically, we describe the microscopic geometry of the domain and write down the microscopic equations and existence results. Section \ref{est_micro} is devoted to obtaining uniform estimates for the microscopic solutions of the app-PNP and PNP. In Section \ref{conv_micro}, we prove strong convergence and two-scale convergence of the solutions of the microscopic PNP. Next, we derive the homogenized PNP equations as well as discuss the uniqueness issue of the homogenized model in Section \ref{homo_model}. Finally, the paper is concluded with an Appendix, where details of some proofs have been added.

\subsection{Function spaces and notations}
Let us introduce the notations and settings that we use in the paper.

(1) Let $(\Omega, \mu)$ be a measure space and $X$ be a Banach space. Then $L^p(\Omega; X)$, $1 \leq p \leq \infty$, denote the standard Lebesgue-Bochner space equipped with the norm $ \displaystyle \lVert f \rVert_{L^p(\Omega; X)} =\left ( \int_{\Omega} \lVert f ( \cdot )\rVert^p_X \,d \mu \right ) ^\frac{1}{p}$, for $1 \leq p < \infty$; if $p = \infty$, the norm is $ \displaystyle \lVert f \rVert_{L^\infty(\Omega; X)} = \esssup_{\omega \in \Omega}\lVert f (\omega) \rVert_X$.

(2) $C_{per} (\overline{Y})$ is the space of continuous functions on $\R^n$ which are also $Y$-periodic, where $Y = (0,1)^n$. 

(3) $W^{1,p}_{per}(Y^f)$, $1<p<\infty$, is the space of functions belonging to the Sobolev space $W^{1,p}(Y^f)$ and having same trace on the opposite faces of the periodicity box $[0,1]^n$. 

(4) The space $W^{1,p}_{per}(Y^f)/ \R$ consists of equivalence classes in $W^{1,p}_{per}(Y^f)$, where the equivalence relation is defined as follows:
$u \sim \upsilon \iff u-\upsilon \ \text{is a constant, where} \  u,\upsilon \in W^{1,p}_{per}(Y^f). $

(5) The continuous dual of a Banach space $X$ is denoted by $X^\prime$. The image of $x \in X$ under $x^\prime \in X^\prime$ is written as a duality pairing $\langle x^\prime, x \rangle_{X^\prime, X}. $

(6) Let $X$, $Y$ be Banach spaces. Then $X \subset Y$ means that $X$ is continuously embedded in $Y$, and the notation $X \subset \subset Y$ implies that $X$ is compactly embedded in $Y$.

(7) For a domain $\Omega \subset \R^n$, its Lebesgue measure is denoted by $|\Omega|$. 

(8) For the dot product involving gradients, we simply use $\nabla u(x) \nabla \upsilon (x)$ instead of $\nabla u(x) \cdot \nabla \upsilon (x)$.

\section{The microscopic problem}\label{micro_prob}
In this section, we precisely formulate the microscopic problem and comment on the properties of microscopic solutions. 
\subsection{Geometry of the microscopic domain}\label{sett_por}

We consider a porous medium occupying a bounded and connected domain $\Omega\subset \mathbb{R}^n, n=2,3$, with $\partial \Omega$ of class $C^3$. The medium has a periodic microstructure, generated with the help of the scaled standard periodicity cell $Y=(0,1)^n$, which consists of a solid part $ Y^s$ and a fluid or pore part $ Y^f$. We assume that $ Y^s$ is an open set such that $\overline{Y^s}\subset Y$ and that the boundary $\Gamma := \partial Y^s$ is of class $C^3$. Furthermore, let $Y^f = Y\setminus \overline{Y^s}$; see also Fig. \ref{fig_domain_PNP} (left). For $k\in \mathbb{Z}^n$, let $Y_k := Y +k$ and $\Gamma_k:= \Gamma + k$. Furthermore, for $j=f, s$, set $Y_k^j:= Y^j + k$.  

Throughout this paper, $\epsilon$ is a sequence of small positive numbers converging to $0$. Now, let 
$
I_\epsilon = \left\{k \in \mathbb{Z}^n \left \vert \epsilon \overline{Y_k} \subset \Omega \right. \right\}.
$
We define the microscopic domain $\Omega_{\epsilon}$ representing the pore part of the porous medium by
$$
\Omega_\epsilon = \Omega \setminus \bigcup_{k \in I_\epsilon} \epsilon \overline { Y_k^s};
$$
see also Fig. \ref{fig_domain_PNP} (right). We remark that the boundary of $ \Omega_\epsilon$ consists of two disjoint parts:
$$
\partial \Omega_\epsilon = \Gamma_\epsilon \cup \partial \Omega,
$$
where 
$
\Gamma_\epsilon := \bigcup_{k \in I_\epsilon} \epsilon \Gamma_k 
$
denotes the boundary of the microscopic solid grains. Finally, we assume that the microscopic domain $\Omega_\epsilon$ is connected with boundary $\partial \Omega_\epsilon$ of class $C^3$.


\subsection{Assumptions on the data}
\begin{itemize}
		\item[(A1)] Let $T$ be an arbitrary positive real number representing the time interval $[0,T]$. 
	\item[(A2)] For the diffusion coefficients, we assume $D_i \in C^2([0,T] \times \overline{\Omega})$ and $m \leq D_i \leq M$, where $i=1, \ldots, P$, for some positive constants $m,M$. 
	\item[(A3)] The boundary data for the electric potential is given by
	\begin{flalign*}
		\xi_\epsilon(x)=
		\begin{cases}
			\epsilon \xi_1 (x, \frac{x}{\epsilon}) &\text{ if $x\in \Gamma_\epsilon$,} \\
			\xi_2(x) &\text{ if $x\in \partial\Omega$,}
		\end{cases}
	\end{flalign*}
	where $\xi_1 (x,y) \in C^2(\overline{\Omega} \times \overline{\Gamma})$ and is periodically extended with respect to $y$ with period $Y$ and $\xi_2 \in C^2(\partial \Omega)$. 
	\item[(A4)]	 For the initial concentrations, we assume  $c^0_i \in C^2(\overline{\Omega})$ with $c^0_i \geq 0$ for $i=1, \ldots, P$.
	\item [(A5)] We assume the following compatibility condition:
	\begin{equation}\label{assum_data}
		\int_{\Omega_\epsilon} \sum_{i=1}^{P} z_i c_i^0 (x) \,dx + \int_{\partial \Omega_\epsilon} \xi_\epsilon(x)\,dS(x) =0.
	\end{equation}
	This compatibility condition has been used in order to prove existence of solutions to the microscopic app-PNP system (\ref{non_dim_app_PNP}) \cite[Proposition 2.1]{Bha22}. Indeed, if we test the equation (\ref{weak_Poisson}) with 1 at $t=0$, we see that this condition is needed.
\end{itemize}
\subsection{Existence of weak solutions for microscopic PNP}
\begin{proposition}\label{ext_micro_prop_st}
	Suppose the assumptions $(A1)-(A5)$ are satisfied. Then there exist $c_{i,\epsilon} \in  L^{\frac{5}{4}}(0,T; W^{1,\frac{5}{4}}(\Omega_\epsilon)) \cap L^{\frac{5}{4}}(0,T; L^\frac{15}{7}(\Omega_\epsilon))$, $c_{i, \epsilon} \geq 0$ a.e. in $(0,T) \times \Omega_\epsilon$ and $\phi_{\epsilon} \in L^\infty(0,T;W^{1,2}(\Omega_\epsilon)) \cap L^2(0,T;W^{2,2}(\Omega_\epsilon)) $ with $\int _{\Omega_\epsilon} \phi_\epsilon (t,x) \,dx   =0$ (for a.e. $t \in (0,T)$) satisfying 
	\begin{eqnarray}\label{Exist_1_micro}
		-	\int_{0}^{T} \int_{\Omega_{\epsilon}} c_{i,\epsilon} \partial_{t} \psi_1 \,dx \,dt+  \int_{0}^{T} \int_{\Omega_\epsilon} ( D_{i} \nabla  c_{i,\epsilon} + D_{i} z_i c_{i,\epsilon} \nabla \phi_\epsilon) \nabla \psi_1  \,dx \,dt \nonumber\\
		= \int_{\Omega_{\epsilon}}c_{i}^0 (x) \psi_1 (0,x) \,dx 
	\end{eqnarray}
	for all $\psi_1 \in C^\infty ([0,T] \times \overline{\Omega}_\epsilon)$ with $\psi_1(T,.) =0$ 
	and 
	\begin{equation}\label{weak_Poisson_PNP}
	\int_{\Omega_\epsilon} \nabla \phi_\epsilon (t) \nabla \psi_2 \,dx  =\int_{\Omega_\epsilon} \sum_{i=1}^{P} z_i c_{i,\epsilon}(t)\psi_2 \,dx   +  \int_{\partial \Omega_\epsilon } \xi_\epsilon \psi_2\,dS(x) 
	\end{equation}
	for all  $\psi_2 \in C^\infty (\overline{\Omega}_\epsilon)$ and for a.e. $t \in (0,T)$. 
\end{proposition}
\begin{proof}
The proof can be found in \cite{Both14}. Although \cite{Both14} deals with the Robin boundary condition for the electric potential, the proof can be simply extended to our setting. We remark that the fact that $ c_{i,\epsilon} \in L^{\frac{5}{4}}(0,T; W^{1,\frac{5}{4}}(\Omega_\epsilon))$ was not shown in \cite{Both14}. However, this can be proved by an application of \cite[(3.8), p. 77]{Lad88} (see Proposition \ref{uni_est_thm_1_st} and Proposition \ref{uni_est_thm_2_st} for details). Consequently, by the Sobolev embedding theorem \cite[Corollary 9.14]{Bre11}, we get $ c_{i,\epsilon} \in 
L^{\frac{5}{4}}(0,T; L^\frac{15}{7}(\Omega_\epsilon))$. The non-negativity of $c_{i,\epsilon}$ and the fact that $\phi_{\epsilon}$ has zero mean value are proved in Remark \ref{mean_value_zero}. 
\end{proof}
\subsection{The microscopic app-PNP}\label{app-PNP_def}
The microscopic approximate PNP (app-PNP) contains an additional nonlinearity in the diffusion term. It consists of the following system of equations. 
\begin{subequations}\label{non_dim_app_PNP}
	\begin{align}
		\partial_{t} c^\eta_{i,\epsilon}(t,x) +\nabla \cdot J^\eta_{i,\epsilon} (t,x) &=0  && \mbox{ in }  (0,T) \times \Omega_\epsilon, \label{non_dim_app_NP_eq}
		\\
		J^\eta_{i,\epsilon} (t,x) \cdot \nu_\epsilon &=0   && \mbox{ on }  (0,T) \times \partial \Omega_\epsilon, 
		\\
		c^\eta_{i,\epsilon}(0,x) &= c^{0}_{i} (x)  && \mbox{ in } \Oe,
		\\
		-   \Delta \phi^\eta_\epsilon (t,x) &=  \sum_{i=1}^{P}  z_i c^\eta_{i,\epsilon}  (t,x)   && \mbox{ in }  (0,T) \times \Omega_\epsilon, 
		\\
		 \nabla \phi^\eta_\epsilon (t,x) \cdot \nu_\epsilon &  = \xi_\epsilon (x)    && \mbox{ on } (0,T) \times \partial \Omega_\epsilon ,
	\end{align}
\end{subequations}
where the total flux $J^\eta_{i,\epsilon}$ of the $i$-th species is given by
\begin{equation*}
	J^\eta_{i,\epsilon}(t, x) = - \left(  D_{i} (t,x) \nabla h^\eta_p\left(c^\eta_{i,\epsilon}(t,x)\right)+  D_{i} (t,x)z_i c^\eta_{i,\epsilon} (t,x) \nabla \phi^\eta_\epsilon (t,x) \right).
\end{equation*}
 The function $h^\eta_p$ is defined by
\begin{equation}\label{def_h}
	h^\eta_p(r)= r+ \eta r^p, \ \ r \geq 0, \ \eta \in (0,\infty), \  p \in [4, \infty).
\end{equation}
Here, the exponent $p$ is a parameter in the nonlinear diffusion $h^\eta_p$. For the rest of this paper, let us fix any $p \geq 4$. 
\subsection{Existence of weak solutions for microscopic app-PNP}
\begin{proposition}\label{prop_exist_app_PNP}
Suppose the assumptions $(A1)-(A5)$ are satisfied. Then there exist $c^\eta_{i,\epsilon} \in L^\infty ((0,T)\times \Omega_\epsilon) \cap L^2(0,T; H^1(\Omega_\epsilon))$, $c^\eta_{i, \epsilon} \geq 0$, a.e. in $(0,T) \times \Omega_{\epsilon}$, with $ \partial_{t} c^\eta_{i,\epsilon} \in L^2(0,T; H^1(\Omega_\epsilon)^\prime)$ and $\phi^\eta_\epsilon \in L^\infty (0,T; W^{2,6} (\Omega_\epsilon)) $  with $\int _{\Omega_\epsilon} \phi^\eta_\epsilon (t,x) \,dx =0$, for all $t \in [0,T]$, satisfying 
\begin{eqnarray}\label{Exist_1}
	\left<\partial_{t} c^\eta_{i,\epsilon} ,\psi_1\right>_{H^1(\Omega_\epsilon)^\prime, H^1(\Omega_\epsilon) }+   \int_{\Omega_\epsilon} \left( D_{i} \nabla h^\eta_p \left(c^\eta_{i,\epsilon}\right) +  D_{i} z_i c^\eta_{i,\epsilon} \nabla \phi^\eta_\epsilon\right) \nabla \psi_1  \,dx= 0,
\end{eqnarray}
for all $\psi_1 \in H^1(\Omega_\epsilon)$ and almost every $t \in (0,T)$, together with the initial condition
\begin{equation}\label{Exist_2}
	c^\eta_{i,\epsilon}(0)= c_i^0 \quad \mbox{ in } L^2(\Omega_\epsilon),
\end{equation}
and 
\begin{equation}\label{weak_Poisson}
	 \int_{\Omega_\epsilon} \nabla \phi^\eta_\epsilon (t) \nabla \psi_2 \,dx = \int_{\Omega_\epsilon} \sum_{i=1}^{P} z_i c^\eta_{i,\epsilon}(t) \psi_2 \,dx + \int_{\partial \Omega_\epsilon } \xi_\epsilon \psi_2\,dS(x)
\end{equation}
for all  $\psi_2 \in H^1(\Omega_\epsilon)$ and all $t \in [0,T]$.
\end{proposition}
\begin{proof}
We refer to \cite[Proposition 2.1]{Bha22}. 
\end{proof}
\begin{remark}\label{non-negativity}
Let us prove the non-negativity of $c_{i,\epsilon}^\eta$, provided the initial concentration $c_i^0$ is non-negative. Let $\widebar{h^\eta_p}$ be the extension of $h^\eta_p$:
\begin{flalign*}
\widebar{h^\eta_p}(r)	=
	\begin{cases}
		r + \eta r^p &\text{ if $r\geq 0$,} \\
		r  &\text{ if $r <0$.}
	\end{cases}
\end{flalign*}
First, we note that the existence result given in Proposition \ref{prop_exist_app_PNP} still holds if we replace $h^\eta_p$ by $\widebar{h^\eta_p}$ in (\ref{Exist_1}). Indeed, this essentially follows from the fact that, for a given $\phi_\epsilon^\eta$, the argument that gives existence of $c_{i,\epsilon}^\eta$ for $h^\eta_p$ in \cite[p. 158]{Both14} also works to give existence of a $c_{i,\epsilon}^\eta$ for $\widebar{h^\eta_p}$. Now, it is enough to prove non-negativity of the $c^\eta_{i,\epsilon}$ satisfying the equation involving $\widebar{h^\eta_p}$. Testing with $\left(c_{i,\epsilon}^\eta\right)^- : = \min\{c_{i,\epsilon}^\eta, 0\}$, we get: 
\begin{eqnarray}\label{non-negativity_pf}
\frac{1}{2} \frac{d}{dt} \int_{\Oe} \left | \left(c_{i,\epsilon}^\eta\right)^- \right |^2 \,dx + \int_{\Oe}  D_i \nabla \widebar{h^\eta_p} (c_{i,\epsilon}^\eta) \nabla \left(c_{i,\epsilon}^\eta\right)^- \,dx \nonumber \\ 
= - \int_{\Oe} D_i z_i \left(c_{i,\epsilon}^\eta\right)^- \nabla \phi_\epsilon ^ \eta \nabla  \left(c_{i,\epsilon}^\eta\right)^- \,dx. &&
\end{eqnarray}
Now, noting 
\begin{equation*}
\int_{\Oe}  D_i \nabla \widebar{h^\eta_p} (c_{i,\epsilon}^\eta) \nabla \left(c_{i,\epsilon}^\eta\right)^- \,dx = \int_{\Oe} D_i \left | \nabla \left(c_{i,\epsilon} ^\eta \right)^- \right |^2 \,dx \geq m \int_{\Oe} \left | \nabla \left(c_{i,\epsilon} ^\eta \right)^- \right |^2 \,dx,
\end{equation*}
using $\phi_\epsilon^\eta \in L^{\infty}(0,T; W^{1,\infty}(\Oe))$ and Young's inequality in the RHS of (\ref{non-negativity_pf}), we have, for any $\delta >0$, that 
\begin{align*}
\frac{1}{2} \frac{d}{dt} \int_{\Oe} \left | \left(c_{i,\epsilon}^\eta\right)^- \right |^2 \,dx + m \int_{\Oe} \left | \left( \nabla c_{i,\epsilon} ^\eta \right) ^- \right |^2 \,dx  \nonumber \\ 
\leq C \left( \delta \int_{\Oe} \left | \nabla \left(c_{i,\epsilon} ^\eta \right)^- \right |^2 \,dx + C_\delta \int_{\Oe} \left |  \left(c_{i,\epsilon} ^\eta \right)^- \right |^2 \,dx \right). 
\end{align*}
Consequently, choosing $C\delta < m$ to absorb the gradient term in the RHS by the LHS, applying Gr{\"o}nwall's inequality and using the initial condition $\left(c_{i,\epsilon} ^\eta \right)^- (0) =0$, we conclude that $\left(c_{i,\epsilon} ^\eta \right)^-  =0$.
\end{remark}
\subsection{Convergence of microscopic app-PNP to microscopic PNP}
\begin{lemma}\label{conv_approx_micro}
Let us fix any $\epsilon >0$. Then there exists a subsequence of $\eta$ (still indexed by $\eta$) such that $c^\eta_{i,\epsilon}$ converges to $c_{i,\epsilon}$ strongly in $L^1((0,T) \times \Oe)$, and $\phi_\epsilon^\eta$ converges to $\phi_\epsilon$ weakly in  $L^q(0,T;W^{1,2}(\Omega_{\epsilon}))$, for all $q \in (1, \infty)$, as $\eta \to 0$. 
\end{lemma}
The proof can be found in \cite[Theorem 1]{Both14}. For the reader's convenience, we add a few details to this proof in Appendix. Furthermore, we show that a different triple of Banach spaces can be used in order to extract a strongly convergent subsequence of $c^\eta_{i,\epsilon}$ converging to $c_{i,\epsilon}$ in $L^1((0,T) \times \Oe)$ as $\eta \to 0$. 
\begin{remark}\label{mean_value_zero}
As a corollary, we have non-negativity of $c_{i,\epsilon}$. Indeed, this follows from the facts that $c_{i,\epsilon}^\eta$ are non-negative (see Remark \ref{non-negativity}) and, up to a subsequence, $c_{i,\epsilon}^\eta \to c_{i,\epsilon}$ as $\eta \to 0$ a.e. in $(0,T)\times \Oe$. Moreover, we have that the electric potential $\phi_{\epsilon}$ satisfying the microscopic PNP has zero mean value, namely $\displaystyle \int_{\Oe} \phi_\epsilon (t,x) \,dt = 0$ for a.e. $t \in (0,T)$. This is a direct consequence of the convergence property of $\phi_\epsilon^\eta$ to $\phi_\epsilon$ from Lemma \ref{conv_approx_micro} and $\phi_\epsilon^\eta$ having zero mean value (by Proposition \ref{prop_exist_app_PNP}).
\end{remark}
	\section{Uniform estimates for microscopic solutions}\label{est_micro}
In this section, we obtain estimates for solutions of the microscopic PNP uniformly w.r.t. the microscopic scale parameter $\epsilon$ (Proposition \ref{uni_est_thm_2_st}). We achieve this by deriving uniform estimates for solutions of the microscopic app-PNP system. It turns out that the estimates for the microscopic app-PNP also play a crucial role in the convergence result for the solutions of the microscopic PNP (see Section 4). The fundamental aspect of deriving the estimates for the app-PNP is that we obtain these bounds uniformly with respect to both the approximation parameter $\eta$ and the microscopic scale parameter $\epsilon$. 
\subsection{Estimates for microscopic app-PNP}\label{subsec_est_appPNP}
	\begin{lemma}[Conservation of mass]\label{lemma_conv_mass}
	 For all $\eta, \epsilon >0$ and for all $t \in [0,T]$, one has 
		\begin{equation*}
\int_{\Oe} c^\eta_{i, \epsilon} (t,x) \,dx =  \int_{\Oe} c^0_{i} (x) \,dx. 
		\end{equation*}
In other words, we have the estimate: $\displaystyle \left \lVert c^\eta_{i, \epsilon}\right\rVert_{C([0,T]; L^1 (\Oe))} \leq C$, where $C >0$ is a constant independent of $\eta$ and $\epsilon$. 
		\end{lemma}
		\begin{proof}
Testing the weak formulation (\ref{Exist_1}) with 1, we get 
\begin{equation*}
\int_{\Oe} c^\eta_{i, \epsilon} (t,x) \,dx = \int_{\Oe} c^0_{i} (x) \,dx \leq \left \lVert c_i^0 \right \rVert_{L^\infty(\Omega_\epsilon)} \left | \Oe \right | \leq C. 
\end{equation*}
		\end{proof}
Next, the energy functional associated with the app-PNP (see \ref{appen_energy_def}) helps us obtain the uniform estimates given in the following two lemmas. 
		\begin{lemma}\label{log_est_st}
There exists a constant $C >0$ independent of $\eta$ and $\epsilon$ such that
\begin{equation*}
\esssup_{t \in (0,T)} \int_{\Oe} \left|c^\eta_{i, \epsilon} (t,x) \log c^\eta_{i, \epsilon} (t,x)\right| \,dx \leq C, \ \ \left \lVert  \phi^\eta_{\epsilon} \right \rVert_{L^\infty (0,T;H^1(\Omega_\epsilon))} \leq C. 
\end{equation*}
		\end{lemma}
		\begin{proof}
Note that, since $\displaystyle \lim_{x \to 0^+} x \log x $ exists, by continuity the domain of $x\log x$ can be extended to $[0,\infty)$. This makes $c^\eta_{i, \epsilon}  \log c^\eta_{i, \epsilon} $ well-defined, even if $c^\eta_{i, \epsilon}$ vanishes. To be more precise, the term $c_{i,\epsilon}^\eta \log c_{i,\epsilon}^\eta$ was obtained as an a.e. limit of $(c_{i,\epsilon}^\eta + \delta_m)\log (c_{i,\epsilon}^\eta + \delta_m)$ as $m \to \infty$, where $\delta_m$ is a sequence of positive numbers converging to $0$ (see \cite[Proposition 3.1]{Bha22}. Hence, if $c_{i,\epsilon}^\eta =0$, we have
\begin{equation*}
 c_{i,\epsilon}^\eta \log c_{i,\epsilon}^\eta = \displaystyle \lim_{m \to \infty} \delta_m \log \delta_m =0.
\end{equation*} 

Now, let us prove the desired estimates. We observe that the constant $C>0$ in the statement of \cite[Proposition 3.1]{Bha22} is independent of $\eta$, $\epsilon$ and $\lambda$. Now, choosing $\alpha = \beta = 0$, $\lambda = 1$ in \cite[Proposition 3.1]{Bha22}, we get 
\begin{equation}\label{log_est_1}
\esssup_{t \in (0,T)} \int_{\Oe} c^\eta_{i, \epsilon} (t,x) \log c^\eta_{i, \epsilon} (t,x) - c^\eta_{i, \epsilon} (t,x) + 1 \,dx \leq C,
\end{equation}
where $C>0$ is independent of $\eta$ and $\epsilon$. Now, using $r \log r -r +1 \geq 0$, for $r \geq 0$, we estimate: 
\begin{eqnarray}
&&\int_{\Oe} \left|c^\eta_{i, \epsilon} (t,x) \log c^\eta_{i, \epsilon} (t,x)\right| \,dx \nonumber \\
&& \leq \int_{\Oe} \left|c^\eta_{i, \epsilon} (t,x) \log c^\eta_{i, \epsilon} (t,x) -c^\eta_{i, \epsilon}(t,x) +1 \right| \,dx + \int_{\Oe} \left|c^\eta_{i, \epsilon} (t,x) - 1 \right| \,dx \nonumber \\ \nonumber \\
&& \leq \int_{\Oe} c^\eta_{i, \epsilon} (t,x) \log c^\eta_{i, \epsilon} (t,x) -c^\eta_{i, \epsilon} (t,x)+1  \,dx + \int_{\Oe} c^\eta_{i, \epsilon} (t,x) \,dx + |\Omega| \nonumber.
\end{eqnarray}
Consequently, (\ref{log_est_1}) and Lemma \ref{lemma_conv_mass} prove that $\displaystyle \esssup_{t \in (0,T)} \int_{\Oe} \left|c^\eta_{i, \epsilon} (t,x) \log c^\eta_{i, \epsilon} (t,x)\right| \,dx \leq C$. 

Next, let us show that the Dirichlet energy of the electric potential, $\displaystyle \esssup_{t \in (0,T)} \int_{\Oe} \left|\nabla \phi^\eta_{\epsilon} (t,x) \right| ^2 \,dx$, is bounded by a constant $C>0$ independent of $\eta$ and $\epsilon$. Indeed, this follows immediately by noting that $r\log r -r +1 \geq 0$, for $r \geq 0$, and choosing $\alpha = \beta = 0$, $\lambda = 1$ in \cite[Proposition 3.1]{Bha22}. Consequently, by the Poincaré inequality for porous medium (Lemma \ref{lemma_mean_value}), we get $\left \lVert  \phi^\eta_{\epsilon} \right \rVert_{L^\infty (0,T;H^1(\Omega_\epsilon))} \leq C$. This completes the proof. 
		\end{proof}
	\begin{lemma}\label{uni_est_prop_st}
There exists a constant $C >0$ independent of $\eta$ and $\epsilon$ such that 
\begin{equation*}
\bigintss_{0}^{T} \bigintss_{\Omega_\epsilon} \frac{\left | \nabla c^\eta_{i,\epsilon }\right |^2 }{c^\eta_{i,\epsilon }  + 1} + \frac{\left | \eta \nabla \left (c^\eta_{i, \epsilon} \right) ^p + z_i c^\eta_{i,\epsilon } \nabla \phi^\eta_{\epsilon}\right|^2}{c^\eta_{i,\epsilon }+1} + \eta \left | \nabla \left (c^\eta_{i,\epsilon }\right)^{\frac{p}{2}}\right |^2 + |\Delta \phi^\eta_{\epsilon}|^2  \,dx \,dt \leq C. 
\end{equation*}
Here, $p$ is the exponent in the non-linear diffusion $h^\eta_p$ (see (\ref{def_h})) appearing in the equation (\ref{Exist_1}).
	\end{lemma}
The proof of Lemma \ref{uni_est_prop_st} for a fixed $\epsilon >0$ can be found in \cite[Lemma 6]{Both14}. However, one can also obtain this estimate independent of $\epsilon$ by using arguments from the proof of  \cite[Lemma 6]{Both14}, together with the applications of extension operator (Lemma \ref{lemma_extension}) and trace inequality for porous medium (Lemma \ref{TraceInequalityLemma}). For completeness, we give the proof in Appendix. 
\begin{proposition}\label{uni_est_thm_1_st}
There exists a constant $C >0$ independent of $\eta$ and $\epsilon$ such that 
\begin{equation*}
\left\lVert\sqrt{c^\eta_{i,\epsilon }+1}\right\rVert_{L^\infty(0,T;L^2(\Omega_\epsilon))} + \left\lVert\sqrt{c^\eta_{i,\epsilon }+1}\right\rVert_{L^{\frac{10}{3}}((0,T) \times \Omega_\epsilon)} +  \left\lVert\nabla c^\eta_{i,\epsilon }\right\rVert_{L^{\frac{5}{4}}((0,T) \times \Omega_\epsilon)} \leq C.
\end{equation*}
\end{proposition}

\begin{proof}
The bound for $\left\lVert\sqrt{c^\eta_{i,\epsilon }+1}\right\rVert_{L^\infty(0,T;L^2(\Omega_\epsilon))} $ directly follows from the conservation of mass (Lemma \ref{lemma_conv_mass}). 

Consequently, using 
	$\nabla \sqrt{c^\eta_{i, \epsilon} + 1} = \frac{1}{2 \sqrt{c^\eta _{i, \epsilon} + 1}} \nabla c^\eta_{i, \epsilon}$
and Lemma \ref{uni_est_prop_st}, we obtain 
\begin{equation}\label{uni_est_thm_1}
\left\lVert\sqrt{c^\eta_{i,\epsilon }+1}\right\rVert_{L^\infty(0,T;L^2(\Omega_\epsilon))} + \left\lVert\mathlarger{\nabla}\sqrt{c^\eta_{i,\epsilon }+1}\right\rVert_{L^2(0,T;L^2(\Omega_\epsilon))}  \leq C. 
\end{equation}
Now, utilizing (\ref{uni_est_thm_1}), the embedding \cite[(3.8), p. 77]{Lad88} and the properties of the extension operator (Lemma \ref{lemma_extension}), we get the required bound for the second term of the theorem:
\begin{equation}\label{uni_est_thm_2}
\left\lVert\sqrt{c^\eta_{i,\epsilon }+1}\right\rVert_{L^{\frac{10}{3}}((0,T) \times \Omega_\epsilon)} \leq C.
\end{equation}
Finally, let $\lambda, r, s > 0$ such that $\frac{r}{\lambda}, \frac{s}{\lambda} \in (1, \infty)$ and $\frac{1}{r} + \frac{1}{s} = \frac{1}{\lambda}$. We make use of Hölder's inequality to get that 
\begin{eqnarray*}
&&\int_{0}^{T} \int_{\Omega_\epsilon} \left | \nabla c^\eta_{i,\epsilon }\right  | ^{\lambda}  \,dx \,dt = \bigintss_{0}^{T} \bigintss_{\Omega_\epsilon} \left | \frac{ \nabla c^\eta_{i,\epsilon }}{\sqrt{c^\eta_{i,\epsilon }+1}}\right | ^{\lambda} \left | \sqrt{c^\eta_{i,\epsilon }+1}\right | ^{\lambda} \,dx \,dt \\
&& \leq \left (\bigintss_{0}^{T} \bigintss_{\Omega_\epsilon} \left | \frac{ \nabla c^\eta_{i,\epsilon }}{\sqrt{c^\eta_{i,\epsilon }+1}}\right | ^{r} \,dx \,dt \right )^{\frac{\lambda}{r}}  \left (\int_{0}^{T} \int_{\Omega_\epsilon} \left | \sqrt{c^\eta_{i,\epsilon }+1}\right | ^{s}\,dx \,dt \right )^{\frac{\lambda}{s}}.
\end{eqnarray*}
Thanks to Lemma \ref{uni_est_prop_st} and (\ref{uni_est_thm_2}), we can choose $r=2, s=\frac{10}{3}$. This gives $\frac{1}{\lambda} = \frac{1}{2}+ \frac{3}{10} = \frac{4}{5}$. 
\end{proof}
\subsection{Estimates for microscopic PNP}
\begin{proposition}\label{uni_est_thm_2_st}
	There exists a constant $C> 0$ independent of $\epsilon$ such that 
	\begin{equation*}
		\lVert c_{i,\epsilon} \rVert_{L^{\frac{5}{3}}((0,T)\times\Omega_\epsilon)} + \lVert \nabla c_{i,\epsilon} \rVert_{L^{\frac{5}{4}}((0,T)\times\Omega_\epsilon)} + \lVert  \phi_{\epsilon} \rVert_{L^\infty(0,T;H^1(\Omega_\epsilon))} \leq C.
	\end{equation*}
\end{proposition}
\begin{proof}
	By Lemma \ref{conv_approx_micro}, we have that, up to a subsequence,  $c^ \eta_{i, \epsilon}$ converges to $c_{i, \epsilon}$ strongly in $L^1((0,T) \times \Omega_\epsilon)$ and $\phi_\epsilon^\eta$ converges to $\phi_\epsilon$ weakly in $L^q(0,T;W^{1,2}(\Omega_{\epsilon}))$, $q \in (1, \infty)$, as $\eta \to 0$. In what follows, we only consider the subsequence of $\eta$ (still indexed by $\eta$) along which this convergence holds. Again, Proposition \ref{uni_est_thm_1_st} guarantees that $\lVert c^\eta_{i, \epsilon} \rVert_{L^{\frac{5}{3}}((0,T)\times\Omega_\epsilon)} \leq C$. Consequently, $c^\eta_{i, \epsilon}$ converges to $c_{i, \epsilon}$ weakly in $L^{\frac{5}{3}}((0,T)\times\Omega_\epsilon)$. Now, by the lower semi-continuity of norm with respect to weak topology, we obtain the required estimate: $ \lVert c_{i,\epsilon} \rVert_{L^{\frac{5}{3}}((0,T)\times\Omega_\epsilon)} \leq C$. 
	
	Now, let us prove the bound for the second term of this theorem. From Proposition \ref{uni_est_thm_1_st}, we have that $\left\lVert\nabla c^\eta_{i,\epsilon }\right\rVert_{L^{\frac{5}{4}}((0,T) \times \Omega_\epsilon)} \leq C.$ Hence, $\nabla c^\eta_{i,\epsilon }$ converges to $\nabla c_{i, \epsilon}$ weakly in $L^{\frac{5}{4}}((0,T) \times \Omega_\epsilon)^n$ and $\left\lVert \nabla c_{i, \epsilon} \right\rVert_{L^{\frac{5}{4}}((0,T) \times \Omega_\epsilon)} \leq C.$  
	
Now, it remains to estimate $\phi_{\epsilon}$. Note that, by Lemma \ref{log_est_st}, we have $\left\lVert \phi^\eta_{\epsilon}\right\rVert_{L^\infty (0,T;H^1(\Omega_\epsilon))} \leq C$. Therefore, as $\eta \to 0$, $\phi^\eta_{\epsilon}$ converges to $\phi_{\epsilon}$ in $L^\infty(0,T,H^1(\Omega_\epsilon))$ in the weak$^\ast$ topology \cite[Corollary 3.30]{Bre11}. Finally, the lower semi-continuity of norm with respect to weak$^\ast$ topology  \cite[Proposition 3.13]{Bre11} yields that $\lVert \phi_{\epsilon}\rVert_{L^\infty (0,T;H^1(\Omega_\epsilon))} \leq C$.
\end{proof}

\subsection{Estimates for microscopic app-PNP under cut-off functions}\label{subsec_cut_off}
Next, we will construct the cut-off functions and establish some properties of them. These cut-off functions play an essential role in obtaining the strong convergence of the microscopic concentrations (see Theorem \ref{thm_st_conv_micro_1}).

Let $K>1$. Let $G_K: \R \to \R$ be a smooth function with the following properties:
\begin{equation}\label{cut_off_1}
G_K (r) =r \ \text{if} \  r \in  \left[-\frac{1}{2},K\right], \  G_K (r) =0 \ \text{if} \  r \in (-\infty,-1) \cup (2K +4, \infty)
\end{equation}
Moreover, $G_K$ satisfies the bounds: 
\begin{equation} \label{cut_off_2}
G_K (r) \leq K+1 \ \forall r \in [0, \infty), \ G_K(r)\in [-1,0] \ \forall r \in (-\infty, 0], \ |G^\prime_K (r)| +  |G^{\prime \prime}_K (r)| \leq C \ \forall r \in \R,
\end{equation}
where $C>0$ is independent of $K$. 

For example, such a $G_K$ can be constructed as follows. Consider the function $F_K: [0, \infty) \to [0, \infty)$ defined as 
	\begin{flalign*}
	F_K(r)=
	\begin{cases}
	r &\text{ if $0\leq r<K+1$,} \\
		K+1  &\text{ if $K+1 \leq r \leq K+2$,} \\
		-r +2K +3 &\text{ if $K+2<r\leq2K+3$,}\\
		0  &\text{ if $r > 2K+3$}.
	\end{cases}
\end{flalign*}
Note that the graph of $F_K$ is the union of four straight lines with slopes $1, 0, -1$ and $0$. Now, the function $G_K (r)$, for $r \geq 0$, can be obtained by mollifying the corners in the graph of $F_K$. Now, to extend the domain of $G_K$ to the negative real line, we consider the function $f(r) = r$, starting from $r = 0$ to $r = -\frac{1}{2}$, and then smoothly merge its graph with the negative real line to get $G_K (r) =0$ for $r <-1$, as desired.   

Denote $V_\epsilon = H^1(\Omega_\epsilon) \cap L^\infty(\Omega_\epsilon)$. Note that $V_\epsilon$ is a Banach space equipped with the norm 
$$
\lVert . \rVert_{V_\epsilon} := \max \{\lVert . \rVert_{H^1(\Omega_\epsilon)}, \lVert . \rVert_{L^\infty(\Omega_\epsilon)}\};
$$
see \cite[p. 1134]{Kal03} if necessary. 
\begin{proposition}\label{th_st_cut_off}
For any $K>1$, the function $G_K$ satisfies the following properties:
\begin{enumerate}
	[label = (\roman*)]
	\item $\partial_{t} G_K\left(c^\eta_{i,\epsilon}\right)  \in L^1 (0,T; V_\epsilon^\prime)$ with 
	\begin{equation}\label{cut_off_st_(i)}
\left\langle \partial_{t} G_K\left(c^\eta_{i,\epsilon}(t)\right), \upsilon \right \rangle_{V_\epsilon^\prime, V_\epsilon} = \left\langle \partial_{t} c^\eta_{i,\epsilon}(t),G_K^\prime \left(c^\eta_{i,\epsilon}(t)\right) \upsilon \right \rangle_{H^1(\Omega_\epsilon)^\prime, H^1(\Omega_\epsilon)} 
	\end{equation}
	for a.e. $t \in (0,T)$ and for all $\upsilon \in V_\epsilon$;
		\item $\left\lVert G_K\left(c^\eta_{i,\epsilon}\right)\right \rVert_{L^2(0,T; V_\epsilon)}\leq  C (K+1+\sqrt{2K+5})$, where $C>0$ is a constant independent of $\eta$ and $\epsilon$;
		\item $\left\lVert \partial_{t} G_K\left(c^\eta_{i,\epsilon}\right)\right \rVert_{L^1 (0,T; V_\epsilon^\prime)}\leq C (2K+4)^{p-1} (2K+5)$, where $C>0$ is a constant independent of $\eta$ and $\epsilon$;
	\item
	\begin{equation*}
	 \int_{t_1}^{t_2} \left\langle \partial_{t} G_K\left(c^\eta_{i,\epsilon}(t)\right), \upsilon \right \rangle_{V_\epsilon^\prime, V_\epsilon} \,dt =  \int_{\Omega_\epsilon} \left(G_K\left(c^\eta_{i,\epsilon}(t_2,x)\right) - G_K\left(c^\eta_{i,\epsilon}(t_1,x)\right)\right)\upsilon (x) \,dx 
	 	\end{equation*} 
	for all $(t_1,t_2) \in [0,T]$ with $t_1 \leq t_2$ and $\upsilon \in V_\epsilon$.
\end{enumerate}
\end{proposition}
\begin{proof}
(i) Let 
\begin{equation}\label{def_lambda}
\langle \Lambda (t) , \upsilon \rangle_{V_\epsilon^\prime, V_\epsilon} :=  \left\langle \partial_{t} c^\eta_{i,\epsilon}(t),G_K^\prime \left(c^\eta_{i,\epsilon}(t)\right) \upsilon \right \rangle_{H^1(\Omega_\epsilon)^\prime, H^1(\Omega_\epsilon)}  \ \forall \upsilon \in V_\epsilon. 
\end{equation}
First, it should be noted that $\Lambda : (0,T) \to V_\epsilon ^\prime$ is strongly measurable. The proof is added in Remark \ref{Rem_meas} in Appendix. 

Next, let us prove (\ref{cut_off_st_(i)}) by density argument. Recall from Proposition \ref{prop_exist_app_PNP} that $c^\eta_{i,\epsilon} \in L^2(0,T;H^1(\Omega_\epsilon))$ and $\partial_{t}c^\eta_{i,\epsilon} \in L^2(0,T;H^1(\Omega_\epsilon)^\prime)$. Therefore, by \cite[Proposition 23.23, (iii)]{Zei90}, there exists a sequence $\rho_m$ in $C^\infty([0,T] \times \overline{\Omega}_\epsilon)$ such that
\begin{eqnarray*} 
	\rho_m \rightarrow c^\eta_{i,\epsilon} \ \text{in $L^2(0,T;H^1(\Omega_\epsilon))$ and} \ \partial_{t}\rho_m \rightarrow \partial_{t}c^\eta_{i,\epsilon} \ \text{in $L^2(0,T;H^1(\Omega_\epsilon)^\prime)$ as $m \to \infty$.}
\end{eqnarray*} 
Let $\upsilon \in V_\epsilon$ and $\psi \in C_0^\infty (0,T)$. Then, up to a subsequence of $\rho_m$, 
\begin{eqnarray}\label{weak_time_deri_calcu}
&& \int_{0}^{T} \left<G_K\left(c^\eta_{i,\epsilon}(t)\right), \upsilon \right>_{V_\epsilon^\prime, V_\epsilon} \psi^\prime (t)\,dt\nonumber\\
&&=\int_{0}^{T} \int_{\Omega_\epsilon} G_K\left(c^\eta_{i,\epsilon}(t,x)\right) \upsilon (x) \psi^\prime (t) \,dx \,dt \nonumber \\
&&= \lim_{m \to \infty} \int_{0}^{T} \int_{\Omega_\epsilon} G_K\left(\rho_m(t,x)\right) \upsilon (x) \psi^\prime (t) \,dx \,dt \nonumber \\
&&=- \lim_{m \to \infty} \int_{0}^{T} \int_{\Omega_\epsilon} G^\prime_K\left(\rho_m(t,x)\right) \partial_{t}\rho_m(t,x)\upsilon (x) \psi(t) \,dx \,dt \nonumber \\
&& = - \int_{0}^{T} \left \langle \partial_{t} c^\eta_{i,\epsilon} (t), G_K^\prime \left(c^\eta_{i,\epsilon}(t)\right) \upsilon \right \rangle_{H^1(\Omega_\epsilon)^\prime, H^1(\Omega_\epsilon)} \psi (t)\,dt \nonumber\\
&&= - \int_{0}^{T} \langle \Lambda (t) , \upsilon \rangle_{V_\epsilon^\prime, V_\epsilon} \psi(t) \,dt.
\end{eqnarray} 
To obtain the fourth equality, we used the fact that $G^\prime_K(\rho_m)\upsilon \psi$ converges weakly to $G_K^\prime(c^\eta_{i,\epsilon}) \upsilon \psi $ in $L^2(0,T;H^1(\Omega_\epsilon))$ (by the generalized dominated convergence theorem \cite[Exercise 20, p. 59]{Fol99}) and $\partial_{t}\rho_m \rightarrow \partial_{t}c^\eta_{i,\epsilon}$ in $L^2(0,T;H^1(\Omega_\epsilon)^\prime)$. Now, (\ref{weak_time_deri_calcu}) implies that $\Lambda = \partial_t G_K\left(c^\eta_{i.\epsilon}\right)$. Consequently, the definition of $\Lambda$ (see \ref{def_lambda}) gives (\ref{cut_off_st_(i)}).

(ii) First, we claim that $t \mapsto \left\lVert G_K\left(c^\eta_{i,\epsilon}(t)\right)\right \rVert_{V_\epsilon}$ is Lebesgue measurable on $(0,T)$. Clearly, $t \mapsto \left\lVert G_K\left(c^\eta_{i,\epsilon}(t)\right)\right \rVert_{H^1(\Omega_\epsilon)}$ is Lebesgue measurable. Also, thanks to \cite[p. 619, Lemma 12.2.2]{Fat99}, we have that $t \mapsto \left\lVert G_K\left(c^\eta_{i,\epsilon}(t)\right)\right \rVert_{L^\infty(\Omega_\epsilon)}$ is Lebesgue measurable. This proves our claim. 

 Next, let us estimate $\left\lVert G_K\left(c^\eta_{i,\epsilon}\right)\right \rVert_{L^2(0,T; V_\epsilon)}$. We have
\begin{eqnarray}\label{th_cut_off_1}
&&\left\lVert G_K\left(c^\eta_{i,\epsilon}\right)\right \rVert_{L^2(0,T; V_\epsilon)}\nonumber\\
&& \leq \left(\int_{0}^{T} \left \lVert G_K\left(c^\eta_{i,\epsilon}(t)\right)\right \rVert_{L^\infty(\Omega_\epsilon)} ^2 \,dt \right)^{\frac{1}{2}} + \left(\int_{0}^{T} \left \lVert G_K\left(c^\eta_{i,\epsilon}(t)\right)\right \rVert_{L^2(\Omega_\epsilon)} ^2 \,dt \right)^{\frac{1}{2}}\nonumber\\
&&+ \left(\int_{0}^{T} \left \lVert \nabla G_K\left(c^\eta_{i,\epsilon}(t)\right)\right \rVert_{L^2(\Omega_\epsilon)} ^2 \,dt \right)^{\frac{1}{2}}.
\end{eqnarray}
Since $0 \leq G_K(r) \leq K+1 \ \forall r \in [0,\infty)$, the first two terms in the RHS of (\ref{th_cut_off_1}) are bounded by $C(K+1)$, where $C>0$ is independent of $\eta$ and $\epsilon$. 

To estimate the last term of (\ref{th_cut_off_1}), we use that $G_K^\prime$ is a bounded function, $G_K (r)$ vanishes if $r >2K+4$ and Lemma \ref{uni_est_prop_st} as follows. Let  $\{c^\eta_{i,\epsilon} \leq 2K+4\}$ denote the set ${\{(t,x)\in (0,T) \times \Omega_\epsilon: c^\eta_{i,\epsilon} (t,x)\leq 2K+4\}}$. We obtain
\begin{eqnarray*}
&&\left(\int_{0}^{T}  \int_{\Omega_\epsilon}\left | G_K^\prime \left(c^\eta_{i,\epsilon}(t,x)\right) \nabla c^\eta_{i,\epsilon}(t,x) \right| ^2 \,dx \,dt \right)^{\frac{1}{2}} \\
&& \leq C \left(\int_{\{c^\eta_{i,\epsilon} \leq 2K+4\}} \left |  \nabla c^\eta_{i,\epsilon} \right| ^2 \,d(t,x) \right)^{\frac{1}{2}} \\
&& \leq C \sqrt{2K+5} \left(\bigintsss_{\{c^\eta_{i,\epsilon} \leq 2K+4\}} \frac{ \left | \nabla c^\eta_{i,\epsilon} \right| ^2}{c^\eta_{i,\epsilon} +1} \,d(t,x) \right)^{\frac{1}{2}} \\
&& \leq C \sqrt{2K+5}. 
\end{eqnarray*}
(iii) Now, we prove the third statement. Let $\upsilon \in V_\epsilon$. Using (i) and (\ref{Exist_1}), we have for a.e. $t$:
\begin{eqnarray*}
&&\left |	\left\langle \partial_{t} G_K\left(c^\eta_{i,\epsilon}(t)\right), \upsilon \right \rangle_{V_\epsilon^\prime, V_\epsilon} \right |= \left | \left\langle \partial_{t} c^\eta_{i,\epsilon}(t),G_K^\prime \left(c^\eta_{i,\epsilon}(t)\right) \upsilon \right  \rangle_{H^1(\Omega_\epsilon)^\prime, H^1(\Omega_\epsilon)} \right | \\
&& \leq \int_{\Omega_\epsilon} \left | D_i \nabla c^\eta_{i,\epsilon} \nabla \left (G^\prime_K \left(c^\eta_{i,\epsilon}\right)\upsilon\right) \right | \,dx +  \int_{\Omega_\epsilon} \left | D_i \eta p \left(c^\eta_{i,\epsilon}\right)^{p-1}  \nabla c^\eta_{i,\epsilon} \nabla \left (G^\prime_K \left(c^\eta_{i,\epsilon}\right)\upsilon\right) \right | \,dx \\
&& +\int_{\Omega_\epsilon} \left | D_i  z_ic^\eta_{i,\epsilon} \nabla \phi_\epsilon^\eta \nabla \left (G^\prime_K \left(c^\eta_{i,\epsilon}\right)\upsilon\right) \right | \,dx \\
&& \leq  \int_{\Omega_\epsilon} \left | D_i \left | \nabla c^\eta_{i,\epsilon}  \right | ^2 G^{\prime \prime}_K \left(c^\eta_{i,\epsilon}\right)\upsilon \right | \,dx + \int_{\Omega_\epsilon} \left | D_i G^{\prime}_K \left(c^\eta_{i,\epsilon}\right) \nabla c^\eta_{i,\epsilon}   \nabla \upsilon \right | \,dx \\
&& +  \int_{\Omega_\epsilon} \left | D_i \eta p \left(c^\eta_{i,\epsilon}\right)^{p-1} \left | \nabla c^\eta_{i,\epsilon}  \right | ^2  G^{\prime \prime}_K \left(c^\eta_{i,\epsilon}\right)\upsilon \right | \,dx \\
&&+  \int_{\Omega_\epsilon} \left | D_i \eta p \left(c^\eta_{i,\epsilon}\right)^{p-1}     G^\prime_K \left(c^\eta_{i,\epsilon}\right)\nabla c^\eta_{i,\epsilon} \nabla \upsilon \right | \,dx\\
&& +  \int_{\Omega_\epsilon} \left | D_iz_i c^\eta_{i,\epsilon}\nabla \phi_\epsilon^\eta  \nabla c^\eta_{i,\epsilon}  G^{\prime \prime}_K \left(c^\eta_{i,\epsilon}\right)\upsilon \right | \,dx +  \int_{\Omega_\epsilon} \left | D_i z_i c^\eta_{i,\epsilon}     G^\prime_K \left(c^\eta_{i,\epsilon}\right)\nabla \phi^\eta_{\epsilon} \nabla \upsilon \right | \,dx.\\
\end{eqnarray*}
Let  $\{c^\eta_{i,\epsilon}(t) \leq 2K+4\}$ denote the set $\{x\in \Omega_\epsilon: c^\eta_{i,\epsilon} (t,x) \leq 2K+4\}$. Now, utilizing the facts that $D_i, G_K^\prime, G_K^{\prime \prime}$ are bounded functions and  $G_K (r)$ vanishes if $r >2K+4$, we proceed with the estimate: 
\begin{eqnarray}\label{functional_expl_1}
	&&\left |	\left\langle \partial_{t} G_K\left(c^\eta_{i,\epsilon}(t)\right), \upsilon \right \rangle_{V_\epsilon^\prime, V_\epsilon} \right | \nonumber\\
	&&\leq C\int_{\{c^\eta_{i,\epsilon}(t) \leq 2K+4\}}  \left | \nabla c^\eta_{i,\epsilon}  \right | ^2 |\upsilon | \,dx + C\int_{\{c^\eta_{i,\epsilon}(t) \leq 2K+4\}} \left |  \nabla c^\eta_{i,\epsilon} \right |  \left |   \nabla \upsilon \right | \,dx \nonumber\\
	&& + C\int_{\{c^\eta_{i,\epsilon}(t) \leq 2K+4\}} \left(c^\eta_{i,\epsilon}\right)^{p-1} \left | \nabla c^\eta_{i,\epsilon}  \right | ^2   | \upsilon  | \,dx +  C\int_{\{c^\eta_{i,\epsilon}(t) \leq 2K+4\}}   \left(c^\eta_{i,\epsilon}\right)^{p-1}     \left |\nabla c^\eta_{i,\epsilon}\right | |\nabla \upsilon  | \,dx\nonumber\\ 
	&& +C \int_{\{c^\eta_{i,\epsilon}(t) \leq 2K+4\}} c^\eta_{i,\epsilon} \left |  \nabla \phi_\epsilon^\eta \right |   \left |   \nabla c^\eta_{i,\epsilon} \right | | \upsilon | \,dx +C \int_{\{c^\eta_{i,\epsilon}(t) \leq 2K+4\}} c^\eta_{i,\epsilon}   \left |    \nabla \phi^\eta_{\epsilon} \right | |\nabla \upsilon | \,dx. \nonumber\\
	&& \leq C (2K+5) \bigintsss_{\{c^\eta_{i,\epsilon}(t) \leq 2K+4\}} \frac{\left | \nabla c^\eta_{i,\epsilon}  \right | ^2}{ c^\eta_{i,\epsilon}  +1}  |\upsilon | \,dx \nonumber\\
	&& + C \sqrt{2K+5}  \bigintsss_{\{c^\eta_{i,\epsilon}(t) \leq 2K+4\}} \frac{\left | \nabla c^\eta_{i,\epsilon}  \right | }{ \sqrt{c^\eta_{i,\epsilon}  +1}} \left |   \nabla \upsilon \right | \,dx \nonumber \\
	&& +C (2K+4)^{p-1} (2K+5) \bigintsss_{\{c^\eta_{i,\epsilon}(t) \leq 2K+4\}} \frac{\left | \nabla c^\eta_{i,\epsilon}  \right | ^2}{ c^\eta_{i,\epsilon}  +1}  |\upsilon | \,dx \nonumber \\
	&& +C (2K+4)^{p-1} \sqrt{2K+5}  \bigintsss_{\{c^\eta_{i,\epsilon}(t) \leq 2K+4\}} \frac{\left | \nabla c^\eta_{i,\epsilon}  \right |}{ \sqrt{c^\eta_{i,\epsilon}  +1}}  | \nabla \upsilon | \,dx \nonumber  \\
	&&+C (2K+4) \sqrt{2K+5} \bigintsss_{\{c^\eta_{i,\epsilon}(t) \leq 2K+4\}} \left |  \nabla \phi_\epsilon^\eta \right |   \frac{\left | \nabla c^\eta_{i,\epsilon}  \right |}{ \sqrt{c^\eta_{i,\epsilon}  +1}}   | \upsilon | \,dx \nonumber \\
	&& +C (2K+4)  \bigintsss_{\{c^\eta_{i,\epsilon}(t) \leq 2K+4\}}    \left |    \nabla \phi^\eta_{\epsilon} \right | |\nabla \upsilon | \,dx.\
\end{eqnarray}
Using $\upsilon \in V_\epsilon$ and H{\"o}lder's inequality, we get that 
\begin{eqnarray}\label{th_cut_off_2}
&&\left\lVert \partial_{t} G_K\left(c^\eta_{i,\epsilon}(t)\right)\right \rVert_{V_\epsilon^\prime}\nonumber\\
&& \leq C (2K+5)  \bigintsss_{\{c^\eta_{i,\epsilon}(t) \leq 2K+4\}} \frac{\left | \nabla c^\eta_{i,\epsilon}  \right | ^2}{ c^\eta_{i,\epsilon}  +1}  \,dx +  C \sqrt{2K+5}  \left( \bigintsss_{\{c^\eta_{i,\epsilon}(t) \leq 2K+4\}} \frac{\left | \nabla c^\eta_{i,\epsilon}  \right | ^2}{ c^\eta_{i,\epsilon}  +1}   \,dx \right)^\frac{1}{2} \nonumber\\
&& +C (2K+4)^{p-1} (2K+5)  \bigintsss_{\{c^\eta_{i,\epsilon}(t) \leq 2K+4\}} \frac{\left | \nabla c^\eta_{i,\epsilon}  \right | ^2}{ c^\eta_{i,\epsilon}  +1}   \,dx \nonumber \\
&& +  C (2K+4)^{p-1} \sqrt{2K+5} \left( \bigintsss_{\{c^\eta_{i,\epsilon}(t) \leq 2K+4\}} \frac{\left | \nabla c^\eta_{i,\epsilon}  \right | ^2}{ c^\eta_{i,\epsilon}  +1}   \,dx \right)^\frac{1}{2} \nonumber\\
&&+C (2K+4) \sqrt{2K+5}\left( \bigintsss_{\{c^\eta_{i,\epsilon}(t) \leq 2K+4\}} \left | \nabla \phi_\epsilon^\eta \right | ^2 \,dx \right)^\frac{1}{2}  \left( \bigintsss_{\{c^\eta_{i,\epsilon}(t) \leq 2K+4\}} \frac{\left | \nabla c^\eta_{i,\epsilon}  \right | ^2}{ c^\eta_{i,\epsilon}  +1}   \,dx \right)^\frac{1}{2}\nonumber\\
&& +C (2K+4) \left( \bigintsss_{\{c^\eta_{i,\epsilon}(t) \leq 2K+4\}} \left | \nabla \phi_\epsilon^\eta \right | ^2 \,dx \right)^\frac{1}{2}. 
\end{eqnarray}
Indeed, we have estimated the second term in the RHS of the last inequality of (\ref{functional_expl_1}) as follows.
\begin{equation*}
\bigintsss_{\{c^\eta_{i,\epsilon}(t) \leq 2K+4\}} \frac{\left | \nabla c^\eta_{i,\epsilon}  \right | }{ \sqrt{c^\eta_{i,\epsilon}  +1}} \left |   \nabla \upsilon \right | \,dx \leq \left( \bigintsss_{\{c^\eta_{i,\epsilon}(t) \leq 2K+4\}} \frac{\left | \nabla c^\eta_{i,\epsilon}  \right | ^2}{ c^\eta_{i,\epsilon}  +1}   \,dx \right)^\frac{1}{2} \lVert  \upsilon \rVert_{V_\epsilon}. 
	\end{equation*} 
The fifth term has been estimated as 
\begin{eqnarray*}
 \bigintsss_{\{c^\eta_{i,\epsilon}(t) \leq 2K+4\}} \left |  \nabla \phi_\epsilon^\eta \right |   \frac{\left | \nabla c^\eta_{i,\epsilon}  \right |}{ \sqrt{c^\eta_{i,\epsilon}  +1}}   | \upsilon | \,dx \leq  \bigintsss_{\{c^\eta_{i,\epsilon}(t) \leq 2K+4\}} \left |  \nabla \phi_\epsilon^\eta \right |   \frac{\left | \nabla c^\eta_{i,\epsilon}  \right |}{ \sqrt{c^\eta_{i,\epsilon}  +1}}   \,dx \lVert \upsilon \rVert_{L^\infty(\Oe)}\\
 \leq \left( \bigintsss_{\{c^\eta_{i,\epsilon}(t) \leq 2K+4\}} \left | \nabla \phi_\epsilon^\eta \right | ^2 \,dx \right)^\frac{1}{2}  \left( \bigintsss_{\{c^\eta_{i,\epsilon}(t) \leq 2K+4\}} \frac{\left | \nabla c^\eta_{i,\epsilon}  \right | ^2}{ c^\eta_{i,\epsilon}  +1}   \,dx \right)^\frac{1}{2} \lVert  \upsilon \rVert_{V_\epsilon}.
	\end{eqnarray*}
The other terms in the RHS of the last inequality of (\ref{functional_expl_1}) have been treated similarly to obtain (\ref{th_cut_off_2}). Consequently, utilizing Lemma \ref{log_est_st} and Lemma \ref{uni_est_prop_st} in (\ref{th_cut_off_2}), we get $\left\lVert \partial_{t} G_K\left(c^\eta_{i,\epsilon}\right)\right \rVert_{L^1 (0,T; V_\epsilon^\prime)}\leq C (2K+4)^{p-1} (2K+5)$. 

(iv) The final statement of the theorem follows from a simple density argument similar to the proof of (i). Hence, the proof is skipped. 
\end{proof}

	\section{Convergence of solutions of microscopic PNP}\label{conv_micro}
	\subsection{Strong convergence}
The strong convergence concerning the microscopic concentrations is proved in the next theorem.  
\begin{theorem}\label{thm_st_conv_micro_1}
For a.e. $t \in (0,T)$, let $\widetilde{c_{i,\epsilon}}(t)$ denote the extension of $c_{i,\epsilon}(t)$ to $\Omega$ as defined by the extension operator (Lemma \ref{lemma_extension}). Then there exists a subsequence, still indexed by $\epsilon$, such that 
	\begin{equation*}
	 \widetilde{c_{i,\epsilon}} \rightarrow c_{i,0} \ \ \ \text{strongly in $L^1(0,T;L^r (\Omega))$, \ for all $r \in \left[2,\frac{15}{7}\right)$},
\end{equation*}	
where the limit function $c_{i,0} \in L^{\frac{5}{4}}(0,T;W^{1,\frac{5}{4}}(\Omega))$ . 
\end{theorem}
\begin{proof}
From Proposition \ref{uni_est_thm_2_st} and Lemma \ref{lemma_extension}, it follows that $\lVert\widetilde{c_{i,\epsilon}}\rVert_{L^{\frac{5}{4}}(0,T;W^{1,\frac{5}{4}}(\Omega))} \leq C$. Hence, there exists $c_{i,0}$ such that, up to a subsequence, $\widetilde{c_{i,\epsilon}}$ converges weakly in $L^{\frac{5}{4}}(0,T;W^{1,\frac{5}{4}}(\Omega))$ to $c_{i,0}$. In what follows, we will consider a subsequence of $\epsilon$ (still indexed by $\epsilon$) along which this weak convergence happens. We will see that this is the subsequence along which the strong convergence, given by the statement of this theorem, holds.  

Again, we have $W^{1,\frac{5}{4}}(\Omega) \subset \subset L^r(\Omega) \subset L^{1}(\Omega)$ for $r \in \left[2,\frac{15}{7}\right)$. Therefore, for all $\alpha >0$, there exists $C(\alpha) >0$ such that 
\begin{equation}\label{strong_conv_1}
\lVert\left(\widetilde{c_{i,\epsilon}} - c_{i,0} \right)(t)\rVert_{L^r(\Omega)} \leq \alpha \lVert\left(\widetilde{c_{i,\epsilon}} - c_{i,0} \right) (t)\rVert_{W^{1,\frac{5}{4}}(\Omega)} + C(\alpha)  \lVert\left(\widetilde{c_{i,\epsilon}} - c_{i,0}\right) (t)\rVert_{L^{1}(\Omega)} 
\end{equation}
for a.e. $t \in (0,T)$ (see \cite[Excercise 6.12]{Bre11}). Integrating in time,  we get 
\begin{eqnarray}\label{strong_conv_2}
\int_{0}^{T} \lVert\left(\widetilde{c_{i,\epsilon}} - c_{i,0} \right)(t)\rVert_{L^r(\Omega)} \,dt &&\leq \alpha \int_{0}^{T} \lVert\left(\widetilde{c_{i,\epsilon}} - c_{i,0} \right) (t)\rVert_{W^{1,\frac{5}{4}}(\Omega)} \,dt \nonumber \\
&&+ C(\alpha) \int_{0}^{T} \lVert\left(\widetilde{c_{i,\epsilon}} - c_{i,0}\right) (t)\rVert_{L^{1}(\Omega)} \,dt.
\end{eqnarray}
Since $\int_{0}^{T} \lVert\left(\widetilde{c_{i,\epsilon}} - c_{i,0} \right) (t)\rVert_{W^{1,\frac{5}{4}}(\Omega)} \,dt $ is uniformly bounded with respect to $\epsilon$, we obtain
\begin{eqnarray}\label{strong_conv_3}
	\int_{0}^{T} \lVert\left(\widetilde{c_{i,\epsilon}} - c_{i,0} \right)(t)\rVert_{L^r(\Omega)} \,dt \leq C_1\alpha + C(\alpha) \int_{0}^{T} \lVert\left(\widetilde{c_{i,\epsilon}} - c_{i,0}\right) (t)\rVert_{L^{1}(\Omega)} \,dt.
\end{eqnarray}
Since we can choose $\alpha >0$ arbitrarily small, note that, to prove $\widetilde{c_{i,\epsilon}}  \rightarrow c_{i,0}$ strongly in $L^1(0,T;L^r(\Omega))$, it is enough to prove that $\widetilde{c_{i,\epsilon}}  \rightarrow c_{i,0}$ strongly in $L^1(0,T;L^{1}(\Omega))$. 

To prove $\widetilde{c_{i,\epsilon}}  \rightarrow c_{i,0}$ strongly in $L^1(0,T;L^{1}(\Omega))$, we will use the Aubin-Lions-Simon lemma (see Lemma \ref{Aubin_Appendix}). Notice that the first assumption of Lemma \ref{Aubin_Appendix} is satisfied; i.e., $\widetilde{c_{i,\epsilon}}$ is bounded in $L^1_{\text{loc}}(0,T;W^{1,\frac{5}{4}}(\Omega))$.

Next, let us show that 
\begin{equation}\label{strong_conv_4}
\int_{0}^{T-h} \int_{\Omega} \left |  \widetilde{c_{i,\epsilon}}(t +h,x)-\widetilde{c_{i,\epsilon}}(t,x)  \right  |  \,dx \,dt \rightarrow 0 \  \text{as} \ h \to 0 \  \text{uniformly in} \ \epsilon, \text{where $h>0$.}
\end{equation}
Note that to prove (\ref{strong_conv_4}), it is enough to show that 
\begin{equation}\label{strong_conv_5}
	\int_{0}^{T-h} \int_{\Omega_\epsilon} \left |  c_{i,\epsilon}(t +h,x)-c_{i,\epsilon}(t,x)  \right  |  \,dx \,dt \rightarrow 0 \  \text{as} \ h \to 0 \  \text{uniformly in} \ \epsilon.
\end{equation}
To obtain this convergence, we use the estimates for the microscopic app-PNP. Let us choose any $\eta$ and $\epsilon$. We have 
\begin{eqnarray}\label{strong_conv_5a}
&&	\int_{0}^{T-h} \int_{\Omega_\epsilon} \left |  c_{i,\epsilon}(t +h,x)-c_{i,\epsilon}(t,x)  \right  |  \,dx \,dt \nonumber \\
&& \leq 	\int_{0}^{T-h} \int_{\Omega_\epsilon}\left |  c_{i,\epsilon}(t +h,x)-c^\eta_{i,\epsilon}(t +h,x)  \right  |  \,dx \,dt \nonumber \\
&&+ \int_{0}^{T-h} \int_{\Omega_\epsilon} \left |  c^\eta_{i,\epsilon}(t +h,x)-c^\eta_{i,\epsilon}(t,x)  \right  |  \,dx \,dt +\int_{0}^{T-h} \int_{\Omega_\epsilon} \left |  c^\eta_{i,\epsilon}(t,x)-c_{i,\epsilon}(t,x)  \right  |\,dx \,dt. \nonumber \\
\end{eqnarray}
In order to estimate the second term in the RHS of (\ref{strong_conv_5a}), we write the set $(0,T-h) \times \Omega_\epsilon$ as a union of the following four disjoint subsets.
\begin{align*}
&\{c^\eta_{i,\epsilon}(t)\leq K, c^\eta_{i,\epsilon}(t+h)\leq K\} \\ 
&:={\{(t,x)\in (0,T-h)\times \Omega_\epsilon: c^\eta_{i,\epsilon}(t,x) \leq K \ \text{and} \ c^\eta_{i,\epsilon}(t+h,x) \leq K\}},\\
&\{c^\eta_{i,\epsilon}(t)>K, c^\eta_{i,\epsilon}(t+h)> K\}\\
& :={\{(t,x)\in (0,T-h)\times \Omega_\epsilon: c^\eta_{i,\epsilon}(t,x) > K \ \text{and} \ c^\eta_{i,\epsilon}(t+h,x) > K\}}, \\
&\{c^\eta_{i,\epsilon}(t)\leq K, c^\eta_{i,\epsilon}(t+h)> K\} \\
&:={\{(t,x)\in (0,T-h)\times \Omega_\epsilon: c^\eta_{i,\epsilon}(t,x) \leq K \ \text{and} \ c^\eta_{i,\epsilon}(t+h,x) > K\}}, \\
&\{c^\eta_{i,\epsilon}(t) > K, c^\eta_{i,\epsilon}(t+h) \leq K\} \\
&:={\{(t,x)\in (0,T-h)\times \Omega_\epsilon: c^\eta_{i,\epsilon}(t,x) > K \ \text{and} \ c^\eta_{i,\epsilon}(t+h,x) \leq K\}},
\end{align*}

where $K>1$ is any fixed number. Hence, 
\begin{eqnarray}\label{strong_conv_6}
\int_{0}^{T-h} \int_{\Omega_\epsilon} \left |  c^\eta_{i,\epsilon}(t +h,x)-c^\eta_{i,\epsilon}(t,x)  \right  |  \,dx \,dt \nonumber\\
=   \underbrace{\int_{\{c^\eta_{i,\epsilon}(t)\leq K, c^\eta_{i,\epsilon}(t+h)\leq K\}} \left |  c^\eta_{i,\epsilon}(t +h,x)-c^\eta_{i,\epsilon}(t,x)  \right  |  \,d(t,x)}_\text{$I$}\nonumber\\
+  \underbrace{\int_{\{c^\eta_{i,\epsilon}(t)>K, c^\eta_{i,\epsilon}(t+h)> K\}} \left | c^\eta_{i,\epsilon}(t +h,x)-c^\eta_{i,\epsilon}(t,x)  \right  | \,d(t,x)}_\text{$II$} \nonumber \\
+ \underbrace{\int_{\{c^\eta_{i,\epsilon}(t)\leq K, c^\eta_{i,\epsilon}(t+h)> K\} } \left | c^\eta_{i,\epsilon}(t +h,x)-c^\eta_{i,\epsilon}(t,x)  \right  | \,d(t,x)}_\text{$III$} \nonumber \\
+  \underbrace{\int_{\{c^\eta_{i,\epsilon}(t) > K, c^\eta_{i,\epsilon}(t+h) \leq K\} } \left | c^\eta_{i,\epsilon}(t +h,x)-c^\eta_{i,\epsilon}(t,x)  \right  | \,d(t,x)}_\text{$IV$}. 
\end{eqnarray}

In what follows, we will estimate each terms of the RHS of (\ref{strong_conv_6}) separately. 

\textbf{Estimate of I:} Recall from Proposition \ref{th_st_cut_off} that, for $\upsilon \in V_\epsilon$, 
\begin{equation}\label{equi_cut_off_1}
	 \int_{t}^{t+h} \left\langle \partial_{s} G_K\left(c^\eta_{i,\epsilon}(s)\right), \upsilon \right \rangle_{V_\epsilon^\prime, V_\epsilon} \,ds =  \int_{\Omega_\epsilon} \left(G_K\left(c^\eta_{i,\epsilon}(t+h,x)\right) - G_K\left(c^\eta_{i,\epsilon}(t,x)\right)\right)\upsilon (x) \,dx  
\end{equation}
Choosing $\upsilon = G_K\left(c^\eta_{i,\epsilon}(t+h)\right) - G_K\left(c^\eta_{i,\epsilon}(t)\right)$ in (\ref{equi_cut_off_1}) and integrating in $t$ from $0$ to $T-h$, we get 
\begin{eqnarray}\label{equi_cut_off_2}
	&&\int_{0}^{T-h}\int_{t}^{t+h} \left\langle \partial_{s} G_K\left(c^\eta_{i,\epsilon}(s)\right), G_K\left(c^\eta_{i,\epsilon}(t+h)\right) - G_K\left(c^\eta_{i,\epsilon}(t)\right) \right \rangle_{V_\epsilon^\prime, V_\epsilon} \,ds \,dt\nonumber \\
	&&=  \int_{0}^{T-h
	}\int_{\Omega_\epsilon} \left|G_K\left(c^\eta_{i,\epsilon}(t+h,x)\right) - G_K\left(c^\eta_{i,\epsilon}(t,x)\right)\right|^2 \,dx  \,dt .
\end{eqnarray}
Next, we estimate the LHS of (\ref{equi_cut_off_2}). Thanks to Fubini, H{\"o}lder and the bounds of $\left\lVert G_K\left(c^\eta_{i,\epsilon}\right)\right \rVert_{L^2(0,T; V_\epsilon)}$, $\left\lVert \partial_{t} G_K\left(c^\eta_{i,\epsilon}\right)\right \rVert_{L^1 (0,T; V_\epsilon^\prime)}$ from Proposition \ref{th_st_cut_off}, we write:

\begin{eqnarray}\label{equi_cut_off_3}
&& \hspace{-.57cm} \int_{0}^{T-h}\int_{t}^{t+h} \left\langle \partial_{s} G_K\left(c^\eta_{i,\epsilon}(s)\right), G_K\left(c^\eta_{i,\epsilon}(t+h)\right) - G_K\left(c^\eta_{i,\epsilon}(t)\right) \right \rangle_{V_\epsilon^\prime, V_\epsilon} \,ds \,dt\nonumber \\
&& \hspace{-.57cm} =\int_{0}^{T-h}\int_{0}^{T} \rchi_{(t,t+h)}(s)\left\langle \partial_{s} G_K\left(c^\eta_{i,\epsilon}(s)\right), G_K\left(c^\eta_{i,\epsilon}(t+h)\right) - G_K\left(c^\eta_{i,\epsilon}(t)\right) \right \rangle_{V_\epsilon^\prime, V_\epsilon} \,ds \,dt\nonumber \\
&& \hspace{-.57cm} \leq\int_{0}^{T-h}\int_{0}^{T} \rchi_{(t,t+h)}(s)\left\lVert \partial_{s} G_K\left(c^\eta_{i,\epsilon}(s)\right)\right\rVert_{V_\epsilon^\prime} \left \lVert G_K\left(c^\eta_{i,\epsilon}(t+h)\right) - G_K\left(c^\eta_{i,\epsilon}(t)\right) \right \rVert_{V_\epsilon} \,ds \,dt\nonumber \\
&& \hspace{-.57cm}  = \int_{0}^{T}\left(\int_{0}^{T-h} \rchi_{(t,t+h)}(s)\left \lVert G_K\left(c^\eta_{i,\epsilon}(t+h)\right) - G_K\left(c^\eta_{i,\epsilon}(t)\right) \right \rVert_{V_\epsilon}  \,dt \right) \left\lVert \partial_{s} G_K\left(c^\eta_{i,\epsilon}(s)\right)\right\rVert_{V_\epsilon^\prime} \,ds\nonumber \\
&& \hspace{-.57cm} \leq\int_{0}^{T}\left( \int_{0}^{T-h}\left | \rchi_{(t,t+h)}(s) \right | ^2 \,dt \right)^\frac{1}{2}\left(\int_{0}^{T-h}\left \lVert G_K\left(c^\eta_{i,\epsilon}(t+h)\right) - G_K\left(c^\eta_{i,\epsilon}(t)\right) \right \rVert_{V_\epsilon} ^2 \,dt \right)^\frac{1}{2} \nonumber\\
&& \hspace{-.57cm} \times \left\lVert \partial_{s} G_K\left(c^\eta_{i,\epsilon}(s)\right)\right\rVert_{V_\epsilon^\prime} \,ds\nonumber \\
&& \hspace{-.57cm} \leq\sqrt{h}\left(\int_{0}^{T-h}\left \lVert G_K\left(c^\eta_{i,\epsilon}(t+h)\right) - G_K\left(c^\eta_{i,\epsilon}(t)\right) \right \rVert_{V_\epsilon} ^2 \,dt \right)^\frac{1}{2}\int_{0}^{T}\left\lVert \partial_{s} G_K\left(c^\eta_{i,\epsilon}(s)\right)\right\rVert_{V_\epsilon^\prime} \,ds\nonumber \\
&& \hspace{-.57cm} \leq C \sqrt{h} (2K+4)^{p-1}(2K+5)(K+1+\sqrt{2K+5}). 
\end{eqnarray}

Again, using the definition of the cut-off function $G_K$ (see (\ref{cut_off_1})), Hölder's inequality and the measure of the set $\Omega_\epsilon$ is bounded by a constant independent of $\epsilon$, we obtain that 
\begin{eqnarray}\label{equi_cut_off_4}
&&\int_{\{c^\eta_{i,\epsilon}(t)\leq K, c^\eta_{i,\epsilon}(t+h)\leq K\}} \left|c^\eta_{i,\epsilon}(t+h,x) - c^\eta_{i,\epsilon}(t,x)\right| \,d(t,x)  \nonumber\\
&&=\int_{\{c^\eta_{i,\epsilon}(t)\leq K, c^\eta_{i,\epsilon}(t+h)\leq K\}} \left|G_K\left(c^\eta_{i,\epsilon}(t+h,x)\right) - G_K\left(c^\eta_{i,\epsilon}(t,x)\right)\right| \,d(t,x) \nonumber \\
&&\leq \int_{0}^{T-h}\int_{\Omega_\epsilon} \left|G_K\left(c^\eta_{i,\epsilon}(t+h,x)\right) - G_K\left(c^\eta_{i,\epsilon}(t,x)\right)\right| \,dx \,dt \nonumber \\
&&\leq C \left( \int_{0}^{T-h}\int_{\Omega_\epsilon} \left|G_K\left(c^\eta_{i,\epsilon}(t+h,x)\right) - G_K\left(c^\eta_{i,\epsilon}(t,x)\right)\right| ^2\,dx \,dt\right)^\frac{1}{2}. 
\end{eqnarray}

Now, using (\ref{equi_cut_off_4}), (\ref{equi_cut_off_3}) in (\ref{equi_cut_off_2}), we obtain:  
\begin{eqnarray}\label{equi_cut_off_5}
&&\int_{\{c^\eta_{i,\epsilon}(t)\leq K, c^\eta_{i,\epsilon}(t+h)\leq K\}} \left|c^\eta_{i,\epsilon}(t+h,x) - c^\eta_{i,\epsilon}(t,x)\right| \,d(t,x)  \nonumber\\
&& \leq C |h|^\frac{1}{4} (2K+4)^{\frac{p-1}{2}}(2K+5)^{\frac{1}{2}}(K+1+\sqrt{2K+5} \ )^{\frac{1}{2}}. 
\end{eqnarray}
\textbf{Estimate of II:}
Thanks to Lemma \ref{log_est_st}, we obtain: 
\begin{eqnarray}\label{strong_conv_8}
&&\int_{\{c^\eta_{i,\epsilon}(t)>K, c^\eta_{i,\epsilon}(t+h)> K\}} \left |  c^\eta_{i,\epsilon}(t +h,x)-c^\eta_{i,\epsilon}(t,x)  \right  |  \,d(t,x) \nonumber \\
&& \leq \int_{\{c^\eta_{i,\epsilon}(t)>K, c^\eta_{i,\epsilon}(t+h)> K\}}  c^\eta_{i,\epsilon}(t +h,x) \,d(t,x) + \int_{\{c^\eta_{i,\epsilon}(t)>K, c^\eta_{i,\epsilon}(t+h)> K\}} c^\eta_{i,\epsilon}(t ,x) \,d(t,x) \nonumber \\
 &&\leq \frac{1}{\log K} \int_{\{c^\eta_{i,\epsilon}(t)>K, c^\eta_{i,\epsilon}(t+h)> K\}} c^\eta_{i,\epsilon}(t +h,x) \log c^\eta_{i,\epsilon}(t +h,x) \,d(t,x) \nonumber \\ 
 &&+ \frac{1}{\log K} \int_{\{c^\eta_{i,\epsilon}(t)>K, c^\eta_{i,\epsilon}(t+h)> K\}} c^\eta_{i,\epsilon}(t,x) \log c^\eta_{i,\epsilon}(t,x) \,d(t,x)\nonumber \\
 && \leq \frac{C}{\log K}. 
\end{eqnarray}
\textbf{Estimates of III \& IV:} Next, we estimate III. The estimate for IV will follow similarly. 
Since $0 \leq {c^\eta_{i,\epsilon}} \leq K$ and $c^\eta_{i,\epsilon}(t+h,x)> K$ on the set $\{c^\eta_{i,\epsilon}(t) \leq K, c^\eta_{i,\epsilon}(t+h)> K\}$, we have 
\begin{eqnarray}\label{strong_conv_13}
&&\int_{\{c^\eta_{i,\epsilon}(t)\leq K, c^\eta_{i,\epsilon}(t+h)> K\}} \left |  c^\eta_{i,\epsilon}(t +h,x)-c^\eta_{i,\epsilon}(t,x)  \right  |  \,d(t,x) \nonumber\\
&& \leq  \int_{\{c^\eta_{i,\epsilon}(t)\leq K, c^\eta_{i,\epsilon}(t+h)> K\}}   c^\eta_{i,\epsilon}(t +h,x) \,d(t,x) \nonumber\\
&& \leq \frac{1}{\log K} \int_{\{c^\eta_{i,\epsilon}(t)\leq K, c^\eta_{i,\epsilon}(t+h)> K\}}   c^\eta_{i,\epsilon}(t +h,x) \log  c^\eta_{i,\epsilon}(t +h,x)  \,d(t,x). 
\end{eqnarray}
Again, by Lemma \ref{log_est_st}, we get 
\begin{eqnarray}\label{strong_conv_14}
&&\int_{\{c^\eta_{i,\epsilon}(t)\leq K, c^\eta_{i,\epsilon}(t+h)> K\}} \left |  c^\eta_{i,\epsilon}(t +h,x)-c^\eta_{i,\epsilon}(t,x)  \right  |  \,d(t,x) \leq \frac{C}{\log K}. 
\end{eqnarray}
Similarly, for the term IV, we obtain
\begin{eqnarray}\label{strong_conv_15}
	&&\int_{\{c^\eta_{i,\epsilon}(t) > K, c^\eta_{i,\epsilon}(t+h)\leq K\} } \left | c^\eta_{i,\epsilon}(t +h,x)-c^\eta_{i,\epsilon}(t,x)  \right  | \,d(t,x) \leq  \frac{C}{\log K}. 
\end{eqnarray}
\textbf{Convergence (\ref{strong_conv_5}):} Now that we have the required estimates, we are ready to prove the uniform convergence given in (\ref{strong_conv_5}).

Using the estimates for I, II, III and IV in (\ref{strong_conv_5a}), we get: 
\begin{eqnarray}\label{strong_conv_15_a}
&&\int_{0}^{T-h} \int_{\Omega_\epsilon} \left |  c_{i,\epsilon}(t +h,x)-c_{i,\epsilon}(t,x)  \right  |  \,dx \,dt \nonumber\\
&& \leq \int_{0}^{T-h} \int_{\Omega_\epsilon} \left |  c_{i,\epsilon}(t +h,x)-c^\eta_{i,\epsilon}(t +h,x)  \right  |  \,dx \,dt    \nonumber\\
&& + \int_{0}^{T-h} \int_{\Omega_\epsilon} \left |  c^\eta_{i,\epsilon}(t ,x)-c_{i,\epsilon}(t ,x)  \right  |  \,dx \,dt + \frac{C}{\log K } + C |h|^\frac{1}{4} W (K),
\end{eqnarray}
where $W (K) = (2K+4)^{\frac{p-1}{2}}(2K+5)^{\frac{1}{2}}(K+1+\sqrt{2K+5} \ )^{\frac{1}{2}}.$ 

Since the constant $C>0$ appearing in (\ref{strong_conv_15_a}) is independent of $\eta$ and $\epsilon$, we note that (\ref{strong_conv_15_a}) holds for all $\eta$ and $\epsilon$. 

Let $\delta >0$ be arbitrary. Fix any $\epsilon >0$. 

Recall that, up to a subsequence, $ c^\eta_{i,\epsilon}$ converges to $ c_{i,\epsilon}$ in $L^1((0,T) \times \Omega_\epsilon)$ as $\eta \to 0$.  So, in (\ref{strong_conv_15_a}), we can choose an $\eta$, which depends on the $\epsilon$, such that 
\begin{equation}\label{strong_conv_16} 
\int_{0}^{T} \int_{\Omega_\epsilon}  \left | c^\eta_{i,\epsilon}(t ,x)-c_{i,\epsilon}(t,x)  \right  | \,dx  \,dt < \frac{\delta}{6}. 
\end{equation}
Again, note that if $K> \exp{\left(\frac{3C}{\delta}\right)}$, then one gets: 
\begin{equation}\label{strong_conv_17}
\frac{C}{\log K} < \frac{\delta}{3}. 
\end{equation}
Therefore, we choose $K = \exp{\left(\frac{3C}{\delta}\right)} +1 $ in (\ref{strong_conv_15_a}). 

Finally, observe that if $h>0$ be such that 
\begin{equation}\label{strong_conv_18}
h < \left(\frac{\delta}{3C W \left( \exp{\left(\frac{3C}{\delta}\right)} +1 \right)}\right)^4,
\end{equation}
then one has: 
\begin{equation}\label{strong_conv_18_a}
 C |h|^\frac{1}{4} W \left( \exp{\left(\frac{3C}{\delta}\right)} +1\right) < \frac{\delta}{3}. 
\end{equation}
Then, using (\ref{strong_conv_16}), (\ref{strong_conv_17}), (\ref{strong_conv_18}), (\ref{strong_conv_18_a}) in (\ref{strong_conv_15_a}), we get that 
\begin{eqnarray}\label{strong_conv_19}
&& \hspace{-3cm}  \int_{0}^{T-h} \int_{\Omega_\epsilon} \left |  c_{i,\epsilon}(t +h,x)-c_{i,\epsilon}(t,x)  \right  |  \,dx \,dt < \delta \nonumber \\
&& \hspace{-3cm}  \text{whenever $0<h < \left(\frac{\delta}{3C W \left( \exp{\left(\frac{3C}{\delta}\right)} +1 \right)}\right)^4$.} 
\end{eqnarray}
Since the $h$ in (\ref{strong_conv_19}) does not depend on $\epsilon$, we conclude that the convergence is uniform in $c_{i,\epsilon}$. 

Hence, (\ref{strong_conv_4}) holds and the Aubin–Lions–Simon lemma implies that there exists a function $g \in L^1((0,T) \times  \Omega)$ such that, up to a subsequence (still indexed by $\epsilon$), $\widetilde{c_{i, \epsilon}}$ converges to $g$ strongly in $L^1((0,T) \times  \Omega)$. Again, we know that $\widetilde{c_{i, \epsilon}}$ converges to $c_{i,0}$ weakly in $L^{\frac{5}{4}}(0,T;W^{1,\frac{5}{4}}(\Omega))$. We will be done, by (\ref{strong_conv_3}), if we prove that $g = c_{i,0}$. This proof follows simply by noting that
\begin{eqnarray*}
 \int_{0}^{T} \int_{\Omega} g \psi \,dx \,dt  = \lim_{\epsilon \to 0} \int_{0}^{T} \int_{\Omega} \widetilde{c_{i,\epsilon}} \psi \,dx \,dt = \int_{0}^{T} \int_{\Omega} c_{i,0} \psi \,dx \,dt \ \ \forall \psi \in L^\infty((0,T) \times \Omega),
\end{eqnarray*}
which gives $g = c_{i,0}$ a.e. in $(0,T) \times \Omega$. 
\end{proof}
\subsection{Two-scale convergence}
Two-scale convergence is a classical tool for passing to the limit in homogenization theory, which was first introduced in \cite{Ngu89} and later developed in \cite{All92}. On these papers, the convergence was discussed in the $L^2$ setting and for the time independent case. The convergence in the $L^p$ setting and for the time independent case was presented in \cite{Luk02}. Since there is no oscillation in the time variable, these methods can simply be extended to the $L^p_tL^q_x$ setting as below. 

\begin{definition}\label{def_ts}
	Let $p,p^\prime,q,q^\prime \in (1,\infty)$ such that $\frac{1}{p}+\frac{1}{p^\prime}=1$ and $\frac{1}{q}+\frac{1}{q^\prime}=1$.	A sequence of functions $u_\epsilon$ in $L^p( 0,T; L^q (\Omega))$ is said to two-scale converge to a function $u_0 \in L^p (0,T;  L^q(\Omega \times Y ))$ if for all $\psi \in L^{p^\prime} (0,T;L^{q^\prime}( \Omega ; C_{per}(\overline{Y})))$, the following holds:
	\begin{eqnarray*}
		\lim_{\epsilon \to 0} \int_{0}^{T} \int_{\Omega}  u_\epsilon (t,x) \psi \left (t,x, \frac{x}{\epsilon}\right) \,dx \,dt = \int_{0}^{T} \int_{\Omega} \int_Y u_0 (t,x,y) \psi (t,x,y) \,dy \,dx \,dt .
	\end{eqnarray*}
\end{definition}
Whenever this holds, we write: $u_\epsilon \rightarrow u_0$ in the two-scale sense--$L^p_tL^q_x$.

Next, let us mention the compactness properties of the two-scale convergence.
\begin{proposition}\label{Prop_t.s.1}
	\begin{enumerate}[label=(\alph*)]
		\item Suppose $p,q \in (1,\infty)$ and $u_\epsilon$ is a bounded sequence in $L^p( 0,T; L^q (\Omega))$. Then there exists a function $u_0 \in L^p (0,T;  L^q(\Omega \times Y ))$ such that, up to a subsequence, $u_\epsilon \rightarrow u_0$ in the two-scale sense--$L^p_tL^q_x$. 
		\item Let $p,q \in (1,\infty)$ and $u_\epsilon$ be a sequence that converges weakly in $L^p( 0,T; W^{1,q} (\Omega))$ to $u_0 \in L^p( 0,T; W^{1,q}(\Omega))$. Then $u_\epsilon \rightarrow u_0$ in the two-scale sense--$L^p_tL^q_x$. Also, there exists a function $u_1 \in L^p (0,T;L^q(\Omega ; W^{1,q}_{per} (Y)/\R))$ such that, up to a subsequence, $\nabla u_\epsilon \rightarrow \nabla  u_0 + \nabla _y  u_1$ in the two-scale sense--$L^p_tL^q_x$. 
	\end{enumerate}
\end{proposition}
\begin{proof}
	We refer to \cite{All92}, \cite{Luk02}; see also \cite[Section 4]{Wie23}, \cite[Appendix B]{Gah24}. 
\end{proof}
For the next proposition, we use this notation: Let $f$ be a function defined on $\Omega_{\epsilon}$. Then $\rchi_{\Omega_\epsilon} f$ denotes the extension of $f$ to $\Omega$ by $0$. 
\begin{proposition}\label{prop_st_conv_micro_2}
There exists a subsequence (still indexed by $\epsilon$) along which the following convergence results are true for all $s \in (1, \infty)$:  
	\begin{align}
		&  \rchi_{\Omega_\epsilon}  c_{i,\epsilon} \rightarrow \rchi_{Y^f} c_{i,0}   &&\text{in the two-scale sense--$L_t^{\frac{5}{4}}L_x^\frac{5}{4}$};\label{THM_CONV_1}\\
		& \rchi_{\Omega_\epsilon} \nabla c_{i,\epsilon}\rightarrow \rchi_{Y^f}  \left (\nabla c_{i,0} + \nabla _y c_{i,1}  \right )  &&\text{in the two-scale sense--$L_t^\frac{5}{4}L_x^\frac{5}{4}$};\label{THM_CONV_2}\\
		&  \rchi_{\Omega_\epsilon} \phi_{\epsilon} \rightarrow \rchi_{Y^f} \phi_{0}   &&\text{in the two-scale sense--$L_t^sL_x^2$}; \label{THM_CONV_3}\\
		& \rchi_{\Omega_\epsilon} \nabla \phi_{\epsilon} \rightarrow \rchi_{Y^f} \left (\nabla \phi_{0} + \nabla _y \phi_{1}  \right )  &&\text{in the two-scale sense--$L_t^sL_x^2$} \label{THM_CONV_4},
			\end{align}	
where $ c_{i,0}$ is the limit function in Theorem \ref{thm_st_conv_micro_1} and $ c_{i,0} \in L^\frac{5}{4}(0,T;W^{1,\frac{5}{4}}(\Omega)) \cap L^\frac{5}{4}(0,T; L^\frac{15}{7}(\Omega))  $, $c_{i,0} \geq 0$ almost everywhere in $(0,T) \times \Omega$;\\
$c_{i,1 }\in L^\frac{5}{4}((0,T)\times \Omega; W_{per}^{1,{\frac{5}{4}}}(Y^f)/\mathbb{R})$; \\
$\phi_0 \in L^\infty(0,T;W^{1,2}(\Omega))$ and $\int_{\Omega} \phi_{0} (t,x) \,dx =0$ for a.e. $t \in (0,T)$; \\
 $\phi_1 \in L^{\infty}(0,T;L^2(\Omega; W_{per}^{1,2}(Y^f)/\mathbb{R}))$. 
\end{proposition}
\begin{proof}
The convergences for a fixed $s \in (1,\infty) $ follow from the uniform estimates for the microscopic solutions (Proposition \ref{uni_est_thm_2_st} ) and compactness properties of two-scale convergence (Proposition \ref{Prop_t.s.1}); see also \cite[Proposition 2.6.1]{Bha23} for more details. In the following, we show the additional statement of this proposition that the subsequence and the limit functions do not depend on $s$.

First, let us comment on the facts that  $\phi_0$ and the subsequence in (\ref{THM_CONV_3}) are independent of $s$. From Theorem \ref{uni_est_thm_2_st} and Lemma \ref{lemma_extension}, we have $\lVert \widetilde{\phi_{\epsilon}}\rVert_{L^\infty(0,T;H^1(\Omega))}$ is uniformly bounded w.r.t. $\epsilon$. Hence, up to a subsequence (still indexed by $\epsilon$), $\phi_{\epsilon}$ converges to a limit $\phi_0$ in $L^\infty(0,T;H^1(\Omega))$ in the weak$^\ast$ topology. Now, we consider the subsequence of $\epsilon$ along which this convergence holds and by the uniqueness of $\phi_0$, we are done. The fact that $\phi_0$ satisfies $\int_{\Omega} \phi_{0} (t,x) \,dx =0$ can be proved by noting that  $\rchi_{\Omega_\epsilon} \phi_{\epsilon}$ converges to $\rchi_{Y^f} \phi_{0}$ in the two-scale sense--$L_t^sL_x^2$ and $\int_{\Omega_\epsilon} \phi_{\epsilon} (t,x) \,dx =0$ for a.e. $t \in (0,T)$ (see Proposition \ref{ext_micro_prop_st}).   

Next, let us discuss the convergence (\ref{THM_CONV_4}). We only consider the  subsequence along which (\ref{THM_CONV_1}, (\ref{THM_CONV_2})), (\ref{THM_CONV_3}) hold. As we showed previously, this subsequence does not depend on $s$. Now, by the uniform estimate for $\nabla \phi_\epsilon$ and the compactness properties of two-scale convergence, we obtain the following:

For any $s \in (1,\infty)$, there exists a further subsequence (still indexed by $\epsilon$) and a limit function $\phi_1 \in  L^{s}(0,T;L^2(\Omega; W_{per}^{1,2}(Y^f)/\mathbb{R}))$ such that

\begin{equation}
\rchi_{\Omega_\epsilon} \nabla \phi_{\epsilon} \rightarrow \rchi_{Y^f} \left (\nabla \phi_{0} + \nabla _y \phi_{1}  \right )  \text{in the two-scale sense--$L_t^sL_x^2$} \label{THM_CONV_5}.
\end{equation}

We would like to show that the subsequence and $\phi_1$ in (\ref{THM_CONV_5}) are independent of $s$. Here, we utilize the macroscopic PDE from Theorem \ref{Thm_st_macro_model}. In fact, we prove in Theorem \ref{Thm_st_macro_model} that $\rchi_{\Omega_\epsilon} \nabla \phi_{\epsilon}$ converges to $\rchi_{Y^f} \left (\nabla \phi_{0} + \nabla _y \phi_{1}  \right )$ in the two-scale sense--$L_t^5L_x^2$, where $\phi_1$ is uniquely given by (\ref{macro_7a}). This uniqueness yields the desired independence on $s \in (1,\infty)$. Moreover, from (\ref{macro_7a}) we conclude that   $\phi_1 \in L^{\infty}(0,T;L^2(\Omega; W_{per}^{1,2}(Y^f)/\mathbb{R}))$. 
\end{proof}
\section{The homogenized problem}\label{homo_model}
\subsection{Derivation of homogenized PNP}
Now, we derive the homogenized PNP system.  For the derivation, we use the strong (Theorem \ref{thm_st_conv_micro_1}) and two-scale convergence (Proposition \ref{prop_st_conv_micro_2}) of the microscopic solutions. We show that the limit functions satisfy the homogenized system weakly, where the weak solutions are understood in the sense given in Definition \ref{def_uniq}.
\begin{theorem}\label{Thm_st_macro_model}
\begin{enumerate} [label=(\alph*)]    	    	
\item The limit functions $c_{i,0}$, $c_{i,1}$ and $\phi_0$ from Proposition \ref{prop_st_conv_micro_2} are weak solutions of the following two-scale homogenized Nernst--Planck model--homogenized Poisson's equation:

		for $i \in \{1,...,P\}$,
	\begin{flalign}\label{Macro_eq_1aa}
& \hspace{-6.cm} \partial_t c_{i,0} - \nabla_x \cdot \left[ \frac{1}{\left | Y^f \right |} \int_{Y^f} D_i \left(\nabla c_{i,0} + \nabla _y  c_{i,1}\right)  \,dy + D_i z_i c_{i,0}A_{hom} \nabla \phi_0 \right  ] =0    &&\hspace{-.5cm}\text{in $(0,T) \times \Omega$}, \nonumber\\
 &\hspace{-5cm} \left[\frac{1}{\left | Y^f \right |}\int_{Y^f} D_i \left(\nabla c_{i,0} + \nabla _y  c_{i,1}\right) \,dy + D_i z_i c_{i,0}A_{hom} \nabla \phi_0 \right  ] \cdot \nu (x) =0   &&\hspace{-.5cm}\text{on $(0,T) \times \partial \Omega$}, \nonumber\\
 \hspace{-1cm} c_{i,0}(0,x) &= c_i^0 (x)  \   \ \ \text{in $\Omega$},\nonumber\\
  \hspace{-1cm} -\nabla_{y} \cdot \left(\nabla c_{i,0} (t,x)+ \nabla _y  c_{i,1}(t,x,y)\right)  &= 0 \ \  \  \hspace{.73cm}\text{in $(0,T) \times \Omega \times Y^f$}, \nonumber\\
 \hspace{-1cm} \left(\nabla c_{i,0} (t,x)+ \nabla _y  c_{i,1}(t,x,y)\right) \cdot \nu (y)  &= 0  \ \ \  \hspace{.73cm}\text{on $(0,T) \times \Omega \times (\partial Y^f \setminus \partial Y) $}, \nonumber\\
 \hspace{-1cm} y \mapsto c_{i,1}(t,x,y) &\ \ \hspace{1.56cm}\text{is $Y$-periodic,} \nonumber \\
\end{flalign} 
			\begin{equation}\label{Macro_eq_2}
			\begin{aligned}
				&&-\nabla_x \cdot \left [A_{hom} \nabla \phi_0 (t,x) \right ]&= \sum_{i=1}^{P} z_i c_{i,0} (t,x) + \frac{1}{\left | Y^f \right | } \int_{\Gamma} \xi_1 (x,y) \,dS(y)   &&\text{in $(0,T) \times \Omega$}, \\
				&& A_{hom} \nabla \phi_0 (t,x)  \cdot \nu &= \frac{1}{\left | Y^f \right | } \xi_2 (x) &&\text{on $(0,T) \times  \partial \Omega$}.
			\end{aligned}
		\end{equation}
		Here $A_{hom}$ is an $n \times n$ matrix, which is constant and elliptic, given by
		\begin{equation}\label{Macro_eq_2a_0}
			A_{hom} e_k = \frac{1}{\left | Y^f \right | } \int_{Y^f} \left (\nabla_y w_k + e_k \right) \,dy, \ \ \text{for $1\leq k \leq n$,} 
		\end{equation}
		where $e_k$ denotes the standard basis vector of $\mathbb{R}^n$ and 
		$w_k \in H^1_{per} (Y^f)/ \mathbb{R}$ is the unique weak solution of the so-called cell problem:
		\begin{equation}\label{Macro_eq_2a}
			\begin{aligned}
				&&-\nabla_y \cdot \left ( \nabla_y w_k (y) +e_k\right) &=0  &&\text{in $Y^f$},\\
				&&\left ( \nabla_y w_k (y) +e_k\right) \cdot \nu  &=0   &&\text{on $\partial Y^f \setminus \partial Y$},\\
				&&y \mapsto w_k (y) &  &&\text{is $Y$-periodic.}
			\end{aligned}
		\end{equation}
		\item Additionally, if $c_{i,1 }\in L^{\frac{5}{4}}((0,T)\times \Omega; W_{per}^{1,2}(Y^f)/\mathbb{R}) \ \forall 1 \leq i \leq P$, then the limit functions $c_{i,0}$, $\phi_0$ from Proposition \ref{prop_st_conv_micro_2} are weak solutions to the following homogenized Nernst--Planck model:
		
		for $i \in \{1,...,P\}$,
		\begin{equation}\label{Macro_eq_1}
			\begin{aligned}
				&& \partial_t c_{i,0} - \nabla_x \cdot \left[ D_i A_{hom} \nabla c_{i,0} + D_i z_i c_{i,0}A_{hom} \nabla \phi_0 \right  ] &=0  &&\text{in $(0,T) \times \Omega$}, \\
				&& \left[ D_i A_{hom} \nabla c_{i,0} + D_i z_i c_{i,0}A_{hom} \nabla \phi_0 \right  ] \cdot \nu &=0   &&\text{on $(0,T) \times \partial \Omega$},\\
				&&c_{i,0}(0,x) &= c_i^0 (x)  &&\text{in $\Omega$}
			\end{aligned}
		\end{equation}
	and the homogenized Poisson's equation (\ref{Macro_eq_2}). 	
		
	\end{enumerate}
	\end{theorem}
\begin{definition}\label{def_uniq}
Let us give the precise definitions of the weak solutions mentioned in Theorem \ref{Thm_st_macro_model}. For the definitions below, let us first define the regularity properties of the weak solutions as follows.

Let $1 \leq i \leq P$, $c_{i,0} \in L^\frac{5}{4}(0,T;W^{1,\frac{5}{4}}(\Omega))$, $c_{i,0} \geq 0$ a.e. in $(0,T) \times \Omega$ and $c_{i,1 }\in L^\frac{5}{4}((0,T)\times \Omega; W_{per}^{1,{\frac{5}{4}}}(Y^f)/\mathbb{R})$. $\phi_0 \in L^\infty(0,T;W^{1,2}(\Omega))$ and $\int_{\Omega} \phi_{0} (t,x) \,dx =0$, for a.e. $t \in (0,T)$. 

\begin{enumerate}[label=(\alph*)]
\item We say the functions $c_{i,0}, c_{i,1}, \phi_0$ are weak solutions of the two-scaled homogenized Nernst--Planck model (\ref{Macro_eq_1aa})--homogenized Poisson's equation (\ref{Macro_eq_2}) if 
\begin{eqnarray}\label{macro_8c}
	&&\hspace{-10mm} -\int_0^T \int_{\Omega} c_{i,0} \partial_t \psi_0 \,dx \,dt + \frac{1}{\left | Y^f \right |} \int_0^T \int_{\Omega} \int_{Y^f} D_i \left(\nabla c_{i,0} + \nabla _y  c_{i,1}\right) \left (\nabla \psi_0  + \nabla _y \psi_1  \right )  \,dy \,dx \,dt \nonumber \\
	&& \hspace{-10mm} + \int_0^T \int_{\Omega}   D_i z_i  c_{i,0} A_{hom} \nabla \phi_0 \nabla \psi_0   \,dx \,dt   = \int_{\Omega} c^0_i \psi_0(0,x) \,dx, 
\end{eqnarray} 
$\left( i \in \{1,...,P\}\right)$ for all $\psi_0 (t,x) \in C^\infty([0,T] \times \overline{\Omega})$ with $\psi_0(T,.)=0$, $\psi_1 \in L^5((0,T) \times \Omega ; W^{1,5}_{per}(Y^f)/ \R )$ and
\begin{eqnarray}\label{macro_8a0}
	&&\int_{0}^{T} \int_{\Omega} A_{hom} \nabla \phi_0 \nabla \psi_0 \,dx \,dt =  \int_{0}^{T} \int_{\Omega} \sum_{i=1}^{P} z_i c_{i,0} \psi_0 \,dx \,dt \nonumber\\
	&&+\frac{1}{|Y^f|}\int_{0}^{T} \int_{\Omega} \int_{\Gamma} \xi_1 (x,y) \psi_0 (t,x) \,dS(y) \,dx \,dt \nonumber\\
	&& + \frac{1}{|Y^f|} \int_{0}^{T} \int_{\partial \Omega} \xi_2 (x) \psi_0 (t,x) \,dS(x) \,dt   
\end{eqnarray}
for all $\psi_2 \in L^5(0,T; W^{1,2}  (\Omega))$.

Note that choosing $\psi_1 =0$ in (\ref{macro_8c}), we obtain a weak formulation of the first three equations of (\ref{Macro_eq_1aa}) satisfied by $c_{i,0}, c_{i,1}, \phi_0$. Whereas choosing $\psi_0 =0$ in (\ref{macro_8c}), we observe that $c_{i,0}, c_{i,1}$ satisfy the last three equations of (\ref{Macro_eq_1aa}) in a weak sense. 
\item $w_k \in H^1_{per} (Y^f)/ \mathbb{R}$ is said to be the weak solution of the cell problem (\ref{Macro_eq_2a}) if 
	\begin{equation}\label{macro_7b}
		\int_{Y^f} \left (\nabla_y w_k(y) + e_k \right) \nabla_{y} \upsilon (y)\,dy =0
\end{equation}
for all $\upsilon \in W^{1,2}_{per}(Y^f)/\R$.
\item We say that $ c_{i,0}$, $\phi_0$ are weak solutions of the homogenized Nernst--Planck model (\ref{Macro_eq_1})--homogenized Poisson's equation (\ref{Macro_eq_2}) if they satisfy the weak formulation of the homogenized Poisson's equation (\ref{macro_8a0}) and the following weak formulation of the Nernst--Planck model:
\begin{eqnarray}\label{macro_11}
	&&-\int_0^T \int_{\Omega} c_{i,0} \partial_t \psi_0 \,dx \,dt + \int_0^T \int_{\Omega}  D_i A_{hom} \nabla c_{i,0} \nabla \psi_0   \,dx \,dt \nonumber \\
	&& +\int_0^T \int_{\Omega}   D_i z_i  c_{i,0} A_{hom} \nabla \phi_0 \nabla \psi_0   \,dx \,dt   = \int_{\Omega} c^0_i \psi_0(0,x) \,dx,
\end{eqnarray} 
$\left( i \in \{1,...,P\}\right)$ for all $\psi_0 (t,x) \in C^\infty([0,T] \times \overline{\Omega})$ with $\psi_0(T,.)=0$.
\end{enumerate}
\end{definition}
	\begin{proof}
Let $\psi_0 (t,x) \in C^\infty([0,T] \times \overline{\Omega})$ with $\psi_0(T,.)=0$ and $\psi_1(t,x,y) \in C^\infty_0((0,T) \times \Omega \times \overline{Y})$, which is $Y$-periodic in $y$. Let us consider a subsequence of $\epsilon$, still denoted by $\epsilon$, along which the convergence results given in Theorem \ref{thm_st_conv_micro_1}, (\ref{THM_CONV_1}), (\ref{THM_CONV_2}), (\ref{THM_CONV_3}) hold and (\ref{THM_CONV_5}) holds for $s=5$. Now considering $\psi_{\epsilon}(t,x):= \psi_0(t,x) + \epsilon \psi_1 \left (t, x, \frac{x}{\epsilon}\right)$ as a test function in (\ref{Exist_1_micro}), we get
\begin{eqnarray}\label{macro_1}
	&&-\int_0^T \int_{\Omega_\epsilon} c_{i, \epsilon} \left [ \partial_t \psi_0 (t,x) + \epsilon \partial_t \psi_1 \left (t, x , \frac{x}{\epsilon}\right) \right ] \,dx \,dt \nonumber \\
	&&+  \int_{0}^{T} \int_{\Omega_\epsilon}  D_i \nabla c_{i, \epsilon} \left[ \nabla \psi_0 (t,x) + \epsilon \nabla_x \psi_1 \left(t,x,\frac{x}{\epsilon}\right) + \nabla_y \psi_1 \left(t,x,\frac{x}{\epsilon}\right) \right] \,dx \,dt \nonumber \\
	&&+  \int_{0}^{T} \int_{\Omega_\epsilon} D_i z_i c_{i, \epsilon} \nabla \phi_\epsilon \left[ \nabla \psi_0 (t,x) + \epsilon \nabla_x \psi_1 \left(t,x,\frac{x}{\epsilon}\right) + \nabla_y \psi_1 \left(t,x,\frac{x}{\epsilon}\right) \right] \,dx \,dt \nonumber \\
	&& = \int_{\Omega_\epsilon} c^0_i (x) \psi_0 (0,x)  \,dx.
\end{eqnarray}
Let us pass to the limit in the third term (which is the nonlinear drift term) in the left-hand side of (\ref{macro_1}). Passing to the limit in other terms will follow from simpler arguments. We have
\begin{align*}
\int_0^T \int_{\Omega_{\epsilon}}  D_i z_i c_{i,\epsilon} \nabla \phi_{\epsilon} \nabla \psi_{\epsilon} \,dx \,dt=& \int_0^T \int_{\Omega_{\epsilon}}   D_i z_i \left(c_{i,\epsilon} - c_{i,0} \right)\nabla \phi_{\epsilon} \nabla \psi_{\epsilon} \,dx  \,dt
	\\
	&+ \int_0^T \int_{\Omega_{\epsilon}}  D_i z_i c_{i,0} \nabla \phi_{\epsilon} \nabla \psi_{\epsilon} \,dx \,dt =: A^1_\epsilon +A_\epsilon^2. 
\end{align*}
Thanks to the essential boundedness of $D_i$ and $\nabla \psi_\epsilon$, bound of $\nabla \phi_\epsilon$ in $L^\infty (0,T; L^2(\Omega_{\epsilon}))$ from Proposition \ref{uni_est_thm_2_st}, strong convergence of $\widetilde{c_{i,\epsilon}}$ in $L^1(0,T;L^r (\Omega))$, $r \in \left[2, \frac{15}{7}\right)$, from Theorem \ref{thm_st_conv_micro_1}, we obtain:
\begin{equation*}
	|A_\epsilon^1| \leq C \left\lVert \widetilde{c_{i,\epsilon}} - c_{i,0} \right \rVert_{L^1(0,T;L^2(\Omega))} \rightarrow 0 \ \ \text{as} \ \epsilon \to 0. 
\end{equation*}
For the term $A_{\epsilon}^2$ we notice that $D_i z_i c_{i,0} \nabla \psi_{\epsilon} \in L^\frac{5}{4}(0,T;L^{\frac{15}{7}}(\Omega; C_{per}(\overline{Y})))^n$ is an admissible test function for the two-scale convergence. Hence, we can use the two-scale convergence result for $\chi_{\Omega_{\epsilon}} \nabla \phi_{\epsilon}$ in $L^{5}_t L^2_x$ from (\ref{THM_CONV_5}) to obtain 
\begin{align*}
	\lim_{\epsilon \to 0 } A_{\epsilon}^2 = \int_0^T \int_{\Omega} \int_{Y^f} D_i  z_i c_{i,0} \left(\nabla \phi_{0} + \nabla_y \phi_{1}\right) \left(\nabla \psi_0 (t,x)+ \nabla_y \psi_1 (t,x,y)\right) \,dy \,dx \,dt.
\end{align*}
Similarly, we can pass to the limit in (\ref{macro_1}) to obtain 
\begin{eqnarray}\label{macro_2}
	&&-\left | Y^f \right | \int_0^T \int_{\Omega} c_{i,0} \partial_t \psi_0 \,dx \,dt \nonumber\\
	&&+  \int_0^T \int_{\Omega} \int_{Y^f} D_i \left(\nabla c_{i,0} + \nabla _y  c_{i,1}\right) \left (\nabla \psi_0  + \nabla _y \psi_1  \right )  \,dy \,dx \,dt \nonumber \\
	&& +\int_0^T \int_{\Omega} \int_{Y^f}   D_i z_i  c_{i,0} \left(\nabla \phi_0 + \nabla_y \phi_1 \right) \left(\nabla \psi_0 + \nabla_y \psi_1 \right)  \,dy \,dx \,dt  \nonumber \\
	&& = \left | Y^f \right |  \int_{\Omega} c^0_i \psi_0(0,x) \,dx.
\end{eqnarray} 
Next, we test the weak formulation of the Poisson's equation (\ref{weak_Poisson_PNP}) with $\psi_{\epsilon}$ and, by similar arguments, we pass to the limit to get that 
\begin{eqnarray}\label{macro_3}
	&&\int_0^T \int_{\Omega} \int_{Y^f} \left ( \nabla \phi_0 + \nabla_y \phi_1 \right) \left( \nabla \psi_0 + \nabla _y \psi_1\right) \,dy \,dx \,dt\nonumber \\
	&& = \left | Y^f \right | \int_{0}^{T} \int_{\Omega} \sum_{i=1}^{P} z_i c_{i,0} \psi_0 \,dx \,dt + \int_{0}^{T} \int_{\Omega} \int_{\Gamma} \xi_1 (x,y) \psi_0 (t,x) \,dS(y) \,dx \,dt  \nonumber \\
	&&  + \int_{0}^{T} \int_{\partial \Omega} \xi_2 (x) \psi_0 (t,x) \,dS(x) \,dt.  
\end{eqnarray}
Note that, by a density argument, (\ref{macro_3}) holds for all $\psi_0 \in L^5(0,T; W^{1,2}(\Omega))$, $\psi_1 \in L^{\frac{5}{4}}(0,T; L^2 ( \Omega; W^{1,2}_{per}(Y^f)/ \R))$. 

Choosing $\psi_0=0$ in (\ref{macro_3}), we get
\begin{equation}\label{macro_7}
	\int_{0}^{T} \int_\Omega \int_{Y^f} \nabla_{y} \phi_1 \nabla_{y}\psi_1 \,dy \,dx \,dt = - \int_0^T \int_{\Omega} \int_{Y^f} \nabla  \phi_0 \nabla_{y}  \psi_1 \,dy \,dx \,dt. 
\end{equation}
Next, we claim that for the limit function $\phi_0$, there exists a unique $\phi_1 \in L^{1}(0,T;L^2(\Omega; W^{1,2}_{per}(Y^f)/\R))$ satisfying (\ref{macro_7}). If possible, suppose that there exist two functions $\phi_1^1, \phi_1^2 \in  L^1(0,T;L^2(\Omega; W^{1,2}_{per}(Y^f)/\R))$ satisfying (\ref{macro_7}). Then, for a.e. $(t,x) \in (0,T) \times \Omega$, we have 
\begin{equation}\label{macro_7_1}
\int_{Y^f}\nabla_{y} \left(\phi_1^1(t,x,y)- \phi_1^2(t,x,y)\right)\nabla_{y}\psi (y) \,dy  = 0  \ \ \  \forall \psi \in W^{1,2}_{per}(Y^f)/\R.
\end{equation}
Choosing $\psi = \phi_1^1(t,x,.)- \phi_1^2(t,x,.)$ in (\ref{macro_7_1}), the proof of our claim is complete.  
 
Now, let us show that, for the limit function $\phi_0$, the unique function $\phi_1$ that satisfies (\ref{macro_7}) is given by  
\begin{equation}\label{macro_7a}
	\phi_{1}(t,x,y) = \displaystyle \sum_{k=1}^n \frac{\partial \phi_{0}(t,x)}{\partial x_k} w_k(y),
\end{equation}
where $w_k$ is the unique weak solution of (\ref{Macro_eq_2a}); that is, for all $\upsilon \in W^{1,2}_{per}(Y^f)/\R$, $w_k$ satisfies (\ref{macro_7b}). From (\ref{macro_7b}), for any $\psi_1 \in L^{\frac{5}{4}}(0,T; L^2 ( \Omega; W^{1,2}_{per}(Y^f)/ \R))$, we have
\begin{equation}\label{macro_8}
	\int_{0}^{T} \int_{\Omega} \int_{Y^f} \sum_{k=1}^{n} \frac{\partial \phi_0}{\partial x_k} (t,x) \left (\nabla_y w_k(y) + e_k \right) \nabla_{y} \psi_1 (t,x,y) \,dy \,dx \,dt =0.
\end{equation} 
Therefore, we conclude that the unique $\phi_1$ satisfying (\ref{macro_7}) is given by (\ref{macro_7a}). Hence, for the limit functions $\phi_0$ and $c_{i,0}$, there exists a unique $\phi_1$, given by (\ref{macro_7a}), that satisfies (\ref{macro_3}). Now, considering $\psi_1=0$ in (\ref{macro_3}) and recalling the definition of the matrix $A_{hom}$ from (\ref{Macro_eq_2a_0}), we obtain that $c_{i,0}, \phi_{0}$ satisfy the weak formulation (\ref{macro_8a0}).

We remark that the ellipticity of $A_{hom}$ is standard in homogenization theory (see, e.g., \cite[Proposition 2.6]{Cio99}). 

Now, let us show that $c_{i,0}, c_{i,1}, \phi_0$ satisfy (\ref{Macro_eq_1aa}) in the weak sense. Consider the drift term in (\ref{macro_2}), which is $\int_0^T \int_{\Omega} \int_{Y^f}   D_i z_i  c_{i,0} \left(\nabla \phi_0 + \nabla_y \phi_1 \right) \left(\nabla \psi_0 + \nabla_y \psi_1 \right)  \,dy \,dx \,dt$.
Note that $\int_0^T \int_{\Omega} \int_{Y^f}   D_i z_i  c_{i,0} \left(\nabla \phi_0 + \nabla_y \phi_1 \right) \nabla \psi_0   \,dy \,dx \,dt = \left | Y^f \right | \int_0^T \int_{\Omega}   D_i z_i  c_{i,0} A_{hom} \nabla \phi_0  \nabla \psi_0    \,dx \,dt$. Again, by (\ref{macro_7}), $\int_0^T \int_{\Omega} \int_{Y^f}   D_i z_i  c_{i,0} \left(\nabla \phi_0 + \nabla_y \phi_1 \right) \nabla_y \psi_1   \,dy \,dx \,dt =0$. Therefore, the drift term becomes $\left | Y^f \right | \int_0^T \int_{\Omega}    D_i z_i  c_{i,0} A_{hom} \nabla \phi_0  \nabla \psi_0    \,dx \,dt.$ Inserting this in (\ref{macro_2}), we see that $c_{i,0}$, $c_{i,1}$, $\phi_{0}$ satisfy the weak formulation of (\ref{Macro_eq_1aa}), as given by (\ref{macro_8c}). This completes the proof of part (a) of the theorem. 

Now, let us prove (b). We claim that for given $c_{i,0}, \phi_{0}$, there exists a unique $c_{i,1}$ satisfying (\ref{macro_8c}). To prove our claim, we use the additional regularity assumption of $c_{i,1}$ with respect to $y$, namely $c_{i,1}(t,x,.) \in W^{1,2}_{per}(Y^f)/\R$. We note that without this extra regularity with respect to $y$, the uniqueness is not clear (see Remark \ref{rem_uniq_corrector} below). 

If possible, suppose there exist two functions $c^1_{i,1}$, $c^2_{i,1}$ satisfying (\ref{macro_8c}). Then we obtain from (\ref{macro_8c}): 
\begin{equation}\label{macro_9}
	\int_{Y^f}   \nabla_{y} \left(c_{i,1}^1(t,x,y)- c_{i,1}^2(t,x,y)\right
	)\nabla_{y}\psi (y) \,dy  = 0  \ \ \  \forall \psi \in W^{1,5}_{per}(Y^f)/\R,
\end{equation}
for a.e. $(t,x) \in (0,T)\times \Omega$. 

Now, taking, by density, $ \psi = c_{i,1}^1(t,x,.)- c_{i,1}^2(t,x,.)
 \in W^{1,2}_{per}(Y^f)/\R$
in (\ref{macro_9}), our claim that $c_{i,1}$ is unique in the space $L^\frac{5}{4}((0,T)\times \Omega; W_{per}^{1,2}(Y^f)/\mathbb{R}))$ is proved. Consequently, by arguments similar to (\ref{macro_7a}), (\ref{macro_8}), we can write this unique $c_{i,1}$ in terms of $c_{i,0}$:
\begin{equation}\label{macro_10}
c_{i,1}(t,x,y) = \displaystyle \sum_{k=1}^n \frac{\partial c_{i,0}(t,x)}{\partial x_k} w_k(y).
\end{equation}
Then inserting the expression for $c_{i,1}$ from (\ref{macro_10}) in (\ref{macro_8c}), recalling the definition of $A_{hom}$ from (\ref{Macro_eq_2a_0}) as well as choosing $\psi_1 =0$ in (\ref{macro_8c}), we have 
\begin{eqnarray*}
&&-\int_0^T \int_{\Omega} c_{i,0} \partial_t \psi_0 \,dx \,dt + \int_0^T \int_{\Omega}  D_i A_{hom} \nabla c_{i,0} \nabla \psi_0   \,dx \,dt \nonumber \\
&& +\int_0^T \int_{\Omega}   D_i z_i  c_{i,0} A_{hom} \nabla \phi_0 \nabla \psi_0   \,dx \,dt   = \int_{\Omega} c^0_i \psi_0(0,x) \,dx,
\end{eqnarray*} 
for all $\psi_0 (t,x) \in C^\infty([0,T] \times \overline{\Omega})$ with $\psi_0(T,.)=0$. This proves that $c_{i,0}$, $\phi_{0}$ are weak solutions of (\ref{Macro_eq_1}). 
\end{proof}
\begin{remark}\label{rem_uniq_corrector}
	When $c_{i,1 }\in L^\frac{5}{4}((0,T)\times \Omega; W_{per}^{1,{\frac{5}{4}}}(Y^f)/\mathbb{R})$, we do not have the characterization of $c_{i,1}$ in terms of $c_{i,0}$ as given by (\ref{macro_10}). This is due to the fact that the uniqueness of a function $u \in W^{1,p}_{per}(Y^f)/ \R$, for $p \in (1,2)$, satisfying
	\begin{equation*}
		\int_{Y^f} \nabla u \nabla \upsilon \,dy = 0 \ \ \forall \upsilon \in W^{1,q}_{per}(Y^f)/\R, \ \ \text{where $\frac{1}{p} + \frac{1}{q} =1$},
	\end{equation*}
	is not obvious (especially in a domain $Y^f$ which is only Lipschitz and containing a hole). 
	However, the characterization (\ref{macro_10}) holds whenever $c_{i,1 }$ has a better regularity with respect to $y$ (as shown in the part (b) of Theorem \ref{Thm_st_macro_model}). 
\end{remark}
\subsection{On uniqueness of the homogenized problem}
In the following proposition, we prove uniqueness of the homogenized model assuming the solution is regular enough. We point out that this is a strong assumption, and it is not obvious that a weak solution $(c_{1,0}, c_{2,0},..., c_{P,0}, \phi_{0})$ of the homogenized model (in the sense of Definition \ref{def_uniq} (c)) enjoys this additional regularity.
\begin{proposition}\label{macro_thm_unique}
Let $V = L^\infty ((0,T) \times \Omega) \cap L^2(0,T;H^1(\Omega))$ and $W = \{u \in L^\infty(0,T;W^{1,\infty}(\Omega)): \int_{\Omega} u (t,x) \,dx =0$ for a.e. $t \in (0,T)\}$. Suppose a weak solution $(c_{1,0}, c_{2,0},..., c_{P,0}, \phi_{0})$ of the homogenized PNP system (\ref{Macro_eq_2})-(\ref{Macro_eq_1}) (in the sense of Definition \ref{def_uniq} (c)) belongs to $V^P \times W.$ Then it is unique in the space $V^P \times W$.
\end{proposition}
\begin{proof}
Suppose there exist two solutions $\left(c^1_{1,0}, c^1_{2,0},..., c^1_{P,0}, \phi^1_{0}\right), \left(c^2_{1,0}, c^2_{2,0},..., c^2_{P,0}, \phi^2_{0}\right) \in V^P \times W$. Then, for a.e. $t \in (0,T)$, we have 
\begin{eqnarray}\label{uniq_1}
&&\left\langle \partial_{t} (c_{i,0}^1 - c_{i,0}^2), c_{i,0}^1 - c_{i,0}^2 \right \rangle_{H^1(\Omega)^\prime, H^1(\Omega)} + \int_{\Omega} D_i A_{hom} \nabla (c_{i,0}^1 - c_{i,0}^2) \nabla (c_{i,0}^1 - c_{i,0}^2) \,dx \nonumber \\
&&= - \int_{\Omega} D_i z_i c_{i,0}^1 A_{hom} \nabla \phi_{0}^1 \nabla (c_{i,0}^1 - c_{i,0}^2) \,dx +  \int_{\Omega} D_i z_i c_{i,0}^2 A_{hom} \nabla \phi_{0}^2 \nabla (c_{i,0}^1 - c_{i,0}^2) \,dx \nonumber \\
&& = - \int_{\Omega} D_i z_i (c_{i,0}^1 - c_{i,0}^2)  A_{hom} \nabla \phi_{0}^1 \nabla (c_{i,0}^1 - c_{i,0}^2) \,dx   \nonumber\\
&&- \int_{\Omega} D_i z_i c_{i,0}^2  A_{hom} \nabla (\phi_{0}^1 -  \phi_{0} ^2) \nabla (c_{i,0}^1 - c_{i,0}^2)  \,dx.  
\end{eqnarray} 
Using ellipticity of $A_{hom}$ and $0 < m < D_i (t,x) \ \forall (t,x) \in [0,T] \times \overline{\Omega}$, we note that there exists a constant $C_1 >0$ such that 
\begin{equation}\label{uniq_2}
C_1 \int_{\Omega}  \left |  \nabla (c_{i,0}^1 - c_{i,0}^2) \right  | ^2   \,dx \leq  \int_{\Omega} D_i A_{hom} \nabla (c_{i,0}^1 - c_{i,0}^2) \nabla (c_{i,0}^1 - c_{i,0}^2) \,dx . 
\end{equation}
Again, utilizing $A_{hom} \nabla \phi_0^1 \in L^\infty ((0,T) \times \Omega)^n$ and Young's inequality, we have for all $\delta >0$: 
\begin{eqnarray}\label{uniq_3}
&&\int_{\Omega} D_i z_i (c_{i,0}^1 - c_{i,0}^2)  A_{hom} \nabla \phi_{0}^1 \nabla (c_{i,0}^1 - c_{i,0}^2) \,dx \nonumber\\
&&\leq  C \delta \int_{\Omega} |\nabla (c_{i,0}^1 - c_{i,0}^2)| ^2 \,dx + C (\delta) \int_{\Omega} |z_i| ^2  |c_{i,0}^1 - c_{i,0}^2| ^2 \,dx . 
\end{eqnarray}
Again, using the estimate 
\begin{equation}\label{uniq_3a}
\lVert \nabla ( \phi_0^1 - \phi_0^2 ) \rVert_{L^2 (\Omega)} \leq C \left \lVert \sum_{i=1}^{P} z_i (c_{i,0}^1 - c_{i,0}^2)\right  \rVert_{L^2 (\Omega)}  
\end{equation}
 and Young's inequality, we get for all $\delta >0$: 
\begin{eqnarray}\label{uniq_4}
	&&\int_{\Omega} D_i z_i c_{i,0}^2  A_{hom} \nabla (\phi_{0}^1 -  \phi_{0} ^2) \nabla (c_{i,0}^1 - c_{i,0}^2)  \,dx \nonumber\\
	&&\leq  C \delta \int_{\Omega} |\nabla (c_{i,0}^1 - c_{i,0}^2)| ^2 \,dx + C (\delta) \int_{\Omega} \sum_{i=1}^{P}  |z_i| ^2  |c_{i,0}^1 - c_{i,0}^2| ^2 \,dx . 
\end{eqnarray} 
Now, we take the sum from $i = 1$ to $P$ in (\ref{uniq_1} ), use (\ref{uniq_2} ), (\ref{uniq_3} ), (\ref{uniq_4} ) and choose $\delta >0$ small enough so that the gradient terms in the RHS can be absorbed by the gradient terms in the LHS to obtain that 
\begin{eqnarray}\label{uniq_5}
\frac{d}{dt} \left (  \sum_{i=1}^{P} \left \lVert z_i (c_{i,0}^1 - c_{i,0}^2 ) \right \rVert^2_{L^2(\Omega)}  \right ) \leq C \sum_{i=1}^{P} \left \lVert z_i (c_{i,0}^1 - c_{i,0}^2 ) \right \rVert^2_{L^2(\Omega)}. 
\end{eqnarray} 
Consequently, Gr{\"o}nwall's lemma yields that $c_{i,0}^1 = c_{i,0}^2$ if $z_i \neq 0$. Note that the uniqueness of $c_{i,0}$ for $z_i = 0$ directly follows from (\ref{uniq_1}). Finally, (\ref{uniq_3a}) implies the uniqueness of $\phi_{0}$. 
\end{proof}
The following lemma states a condition that guarantees convergence of the whole sequence of microscopic solutions, instead of just a subsequence. Its proof is standard.
\begin{lemma}\label{uniq_cor}
Suppose the weak solution $(c_{1,0}, c_{2,0},..., c_{P,0}, \phi_{0})$, in the sense of Definition \ref{def_uniq} (c), of the homogenized Poisson--Nernst--Planck system (\ref{Macro_eq_2})-(\ref{Macro_eq_1}) is unique. Additionally, assume that $c_{i,1 }\in L^{\frac{5}{4}}((0,T)\times \Omega; W_{per}^{1,2}(Y^f)/\mathbb{R})$ for all $1 \leq i \leq P$. Then all the convergence results given in Theorem \ref{thm_st_conv_micro_1} and Proposition \ref{prop_st_conv_micro_2} are valid for the whole sequence.
\end{lemma}
Of course, the uniqueness of the weak solution to the homogenized model is highly non-trivial. However, if one can show that, possibly with some specific and regular enough data, the PNP system has a regularizing effect, then by Proposition \ref{macro_thm_unique} the uniqueness can be obtained. Here, by regularizing effect, we mean that any weak solution can be ultimately shown to live in the more regular space $V^P \times W$. This is not obvious and remains an open issue.
\begin{remark}
Let $\mathcal{M}_\epsilon^\eta$ denote the microscopic app-PNP (\ref{non_dim_app_PNP}), $\mathcal{M}_\epsilon^0$ the microscopic PNP (\ref{non_dim_PNP}), $\mathcal{M}_0^\eta$ the homogenized app-PNP ((4.7)-(4.8) in \cite{Bha22}) and $\mathcal{M}_0^0$ the homogenized PNP ((\ref{Macro_eq_1aa})-\ref{Macro_eq_2})). In this paper, we started from the microscopic PNP $\mathcal{M}_\epsilon^0$ and performed a homogenization to obtain the homogenized PNP $\mathcal{M}_0^0$. Mathematically, we write this as 
\begin{equation}\label{lim_consecutive}
\displaystyle \lim_{\epsilon \to 0} \mathcal{M}_\epsilon^0 = \mathcal{M}_0^0.
\end{equation}
Again, noting $\displaystyle \lim_{\eta \to 0} \mathcal{M}_\epsilon^\eta = \mathcal{M}_\epsilon^0$, the expression (\ref{lim_consecutive}) becomes $\displaystyle \lim_{\epsilon \to 0} \lim_{\eta \to 0} \mathcal{M}_\epsilon^\eta = \mathcal{M}_0^0.$ Now, it is interesting to ask if we get the same homogenized model $\mathcal{M}_0^0$ after switching the limits. Namely, the question is if $\displaystyle \lim_{\eta \to 0}  \lim_{\epsilon \to 0} \mathcal{M}_\epsilon^\eta= \displaystyle \lim_{\epsilon \to 0} \lim_{\eta \to 0}   \mathcal{M}_\epsilon^\eta = \mathcal{M}_0^0$ holds. Here, we do not attempt to answer this question rigorously. Rather, let us mention what we expect to hold, at least heuristically. Since $A_{hom}$ is a constant, elliptic matrix and we have uniform estimates both with respect to $\epsilon$ and $\eta$ (see Subsection \ref{subsec_est_appPNP}), we expect that we can pass to the limit $\eta \to 0$ in $\mathcal{M}_0^\eta$ similarly to \cite{Both14} to obtain $\displaystyle \lim_{\eta \to 0} \mathcal{M}_0^\eta =\displaystyle \lim_{\eta \to 0} \lim_{\epsilon \to 0} \mathcal{M}_\epsilon^\eta = \mathcal{\widetilde{M}}_0^0$, where $\mathcal{\widetilde{M}}_0^0$ is given by (\ref{Macro_eq_2})-(\ref{Macro_eq_1}). Clearly, the homogenized model $\mathcal{\widetilde{M}}_0^0$ is an improved and stronger version of $\mathcal{M}_0^0$. However, $\mathcal{M}_0^0$ becomes $\mathcal{\widetilde{M}}_0^0$ if $c_{i,1}$ satisfies additional regularity with respect to $y$ (as shown in Theorem \ref{Thm_st_macro_model} (b)).
\end{remark}
\begin{appendix}\label{appendix}
	\section{Appendix}
\begin{lemma}[Extension operator]\label{lemma_extension}
	Let $q \in [1, \infty)$. Then there exists a linear extension operator \,  $\widetilde{\cdot}: W^{1,q}(\Omega_\epsilon) \to W^{1,q}(\Omega)$ such that, for all $u_\epsilon \in W^{1,q}(\Omega_\epsilon)$, we have:
	\begin{equation}\label{lemm_extension}
		\lVert \widetilde{u}_\epsilon \rVert_{L^q(\Omega)} \leq C \lVert u_\epsilon \rVert_{L^q(\Omega_\epsilon)} , \ 
		\lVert \nabla \widetilde{u}_\epsilon \rVert_{L^q(\Omega)} \leq C \lVert \nabla u_\epsilon \rVert_{L^q(\Omega_\epsilon)},
	\end{equation}
	where $C >0$ is a constant independent of $\epsilon$. 
	
	Furthermore, if  $u_\epsilon \in W^{1,q}(\Omega_\epsilon) \cap L^\infty (\Omega_\epsilon)$, then $\widetilde{u}_\epsilon \in L^\infty (\Omega)$ and 
	$$
	\lVert \widetilde{u}_\epsilon \rVert _{L^\infty (\Omega)} \leq C \lVert u_\epsilon \rVert_{L^\infty(\Omega_\epsilon)},
	$$ 	where $C >0$ is a constant independent of $\epsilon$.
	
 Moreover, if $u_\epsilon \in W^{1,q}(\Omega_\epsilon)$ is a non-negative function, so is the extension $\widetilde{u}_\epsilon$. 
\end{lemma}
\begin{proof}
See \cite{Cio79}, \cite[Lemma A.3]{Bha22}. 
\end{proof}
\begin{lemma}[Poincaré inequality for porous medium]\label{lemma_mean_value}
	For all $u_\epsilon \in H^1(\Omega_\epsilon)$ with $\frac{1}{|\Omega_\epsilon|} \int_{\Omega_\epsilon} u_\epsilon \,dx =0$, we have
	\begin{equation}\label{Lemma_mean_value}
		\lVert u_\epsilon \rVert _{L^2(\Omega_\epsilon)} \leq C \lVert \nabla u_\epsilon \rVert_{L^2(\Omega_\epsilon)},
	\end{equation}
	where $C >0$ is a constant independent of $\epsilon$. 
\end{lemma}
\begin{proof}
See \cite[Lemma A.4]{Bha22}. 
\end{proof}
\begin{lemma}[Trace inequality for porous medium]\label{TraceInequalityLemma}
Let $q \in [1,\infty)$. Then, for all $u_{\epsilon} \in W^{1,q}(\Omega_{\epsilon})$,  it holds that
	\begin{align*}
		\epsilon\Vert u_{\epsilon} \Vert_{L^q(\Gamma_{\epsilon})}^q  \le  C\left( \Vert u_{\epsilon}\Vert_{L^q(\Omega_{\epsilon})}^q + \epsilon^q  \Vert \nabla u_{\epsilon} \Vert^q_{L^q(\Omega_{\epsilon})}\right),
	\end{align*}
	where $C >0$ is a constant independent of $\epsilon$. 
\end{lemma}
\begin{proof}
The proof is based on decomposing the domain into cells of order $\epsilon$, change of variable in the integration and the standard trace-inequality in the reference element. Such arguments being standard in homogenization theory, the proof is skipped. 
\end{proof}
\begin{lemma}[Aubin-Lions-Simon lemma]\label{Aubin_Appendix}
Let $\Omega \subset \R^n$, $n \in \{2,3\}$, be a bounded Lipschitz domain. Suppose $\{f_k\} \subset L^1(0,T;L^1(\Omega))$ be a sequence of functions such that it is bounded in $L^1_{\text{loc}}(0,T;W^{1,\frac{5}{4}}(\Omega))$ and $\displaystyle \int_{0}^{T-h} \int_{\Omega} \left | f_k (t+h,x) - f_k (t,x) \right|  \,dx \,dt \to 0$ as $h \to 0 \ (h>0)$ uniformly in $k$. Then, up to a subsequence, $f_k$ converges to $f$ strongly in $L^1(0,T;L^1(\Omega))$ for some $f \in L^1(0,T;L^1(\Omega))$.
\end{lemma}
\begin{proof}
Take $X=W^{1,\frac{5}{4}}(\Omega), B = L^{1}(\Omega), p =1$ in \cite[Theorem 3]{Sim87}.
\end{proof}
\begin{remark}[Strong measurability of $\Lambda$]\label{Rem_meas}
Here, we show that $\Lambda$ (given in the proof of Proposition \ref{th_st_cut_off}, (i)) is a strongly measurable function from $(0,T)$ to $V_\epsilon^\prime$. Let $\upsilon \in V_\epsilon$. Using $G_k^\prime \left(c^\eta_{i,\epsilon}(t)\right) \upsilon$ as a test function in (\ref{Exist_1}), we note that 
\begin{equation}\label{rem_uni_1}
\langle \Lambda (t), \upsilon \rangle _{V_\epsilon^\prime, V_\epsilon} = \int_{\Oe} R (t,x) \upsilon (x) \,dx + \sum_{i=1}^{n}  \int_{\Oe} S^i (t,x) \frac{\partial \upsilon (x)}{\partial x_i} \,dx
\end{equation}
for some functions $R \in L^1(0,T; L^1 (\Oe)), \ S^ i \in L^1(0,T; L^2 (\Oe)) \ (1 \leq i \leq n)$. 

Now, there exist sequences of simple functions $\displaystyle r_k (t) =   \sum_{j=1}^{l (k)}  \rchi_{E_j} (t)  \alpha _ j$, $\displaystyle s^i_k (t) =   \sum_{j=1}^{m (k)}  \rchi_{F ^i_j} (t)  \beta^i _ j$, where $\alpha_j \in L^1 (\Oe)$, $ \beta^i _ j \in L^2 (\Oe)$, $\rchi_{E_j}, \rchi_{F ^i_j}$ are Lebesgue measurable subsets of $[0,T]$, such that $r_k (t) \to R (t)$ in $L^1(\Omega_\epsilon)$ and $s^i_k (t) \to S^i(t)$ in $L^2(\Omega_\epsilon)$, as $k \to \infty$, for a.e. $t \in (0,T)$. Again, observe that $\alpha _ j, \beta^i _ j \in V_\epsilon^\prime$ in the sense that 
\begin{equation*}
\langle  \alpha _ j, \upsilon \rangle _{V_\epsilon^\prime, V_\epsilon} =  \int_{\Oe} \alpha _ j (x) \upsilon (x) \,dx, \ \langle  \beta^i _ j , \upsilon \rangle _{V_\epsilon^\prime, V_\epsilon} =  \int_{\Oe}  \beta^i _ j  (x)  \frac{\partial \upsilon (x)}{\partial x_i} \,dx.
\end{equation*}
Consequently, by (\ref{rem_uni_1}), we get:
\begin{equation*}
\left \Vert \Lambda (t) - r_k (t) -  \sum_{i=1}^{n} s^i_k (t) \right \rVert_{V_\epsilon^\prime} \leq \lVert R (t) - r_k (t) \rVert_{L^1(\Omega_\epsilon)} + \sum_{i=1}^{n} \lVert S ^i (t) - s^i_k (t) \rVert_{L^2(\Omega_\epsilon)} \to 0, 
\end{equation*}
as $k \to \infty$, for a.e. $t \in (0,T)$.
\end{remark}
\begin{lemma}\label{appen_energy_lemma}
	Let us consider the app-PNP system with nonlinear diffusion $h_p$, $p\geq 4$, as defined in (\ref{def_h}). For simplicity, we denote the corresponding solution by $\left(c^\eta_{i,\epsilon}, \phi^\eta_\epsilon\right)$, instead of $\left(c^{\eta,p}_{i,\epsilon},\phi^{\eta,p}_{i,\epsilon}\right)$. Then the corresponding energy functional
\begin{equation}\label{appen_energy_def}
V_\epsilon^\eta (t) := \int_{\Omega_\epsilon} \left | \nabla \phi_\epsilon^\eta (t,x) \right | ^2  \,dx  + {\displaystyle \sum_{i=1}^{P}} \int_{\Omega_\epsilon} c^\eta_{i,\epsilon} \log c^\eta_{i,\epsilon} (t,x) - c^\eta_{i,\epsilon}(t,x) + 1 + \frac{\eta}{p-1}  \left ( c^\eta_{i,\epsilon}   (t,x) \right ) ^ p \,dx
\end{equation}
satisfies 
\begin{equation}\label{appen_energy_est}
V_\epsilon^\eta (t) \leq C \ \ \ \forall t \in [0,T],
\end{equation}
for some constant $C>0$ independent of $\eta$ and $\epsilon$. 

Consequently, one gets 
	\begin{equation}\label{eta_blow_up_eq}
	\left \lVert c^\eta_{i,\epsilon} \right \rVert_{L^\infty(0,T;L^p(\Omega_\epsilon))} \leq \frac{C}{\eta} (p-1), 
	\end{equation}
for some constant $C>0$ independent of $\eta$ and $\epsilon$. 
\begin{proof}
The bound (\ref{appen_energy_est}) is proved in \cite[Proposition 3.1]{Bha22}. Now, (\ref{eta_blow_up_eq}) follows from (\ref{appen_energy_est}) and noting that $r\log r -r + 1 \geq 0$ for $r \geq 0$.
\end{proof}
\end{lemma}
\begin{proof}[Proof of Lemma \ref{uni_est_prop_st}]
First, let us go back to the proof of \cite{Bha22}[Proposition 3.1]. There, the term $\frac{\left | J^\eta_{i,\epsilon} \right |^2}{D_i c^\eta_{i,\epsilon}}$ was obtained as an almost everywhere limit of $\frac{\left | J^\eta_{i,\epsilon} \right |^2}{D_i \left(c^\eta_{i,\epsilon} + \delta_m\right)}$ as $m \to \infty$. Let 
$$
A := \{(t,x)\in (0,T) \times \Omega_{\epsilon}: c^\eta_{i,\epsilon} (t,x) = 0\}. 
$$ 
Observe that $\frac{\left | J^\eta_{i,\epsilon} \right |^2}{D_i \left(c^\eta_{i,\epsilon} + \delta_m\right)}$ vanishes a.e. on $A$ \cite[Chapter 5, Problem 18. (c)]{Eva10}. So, 
	\begin{flalign*}
\lim_{m \to \infty} \frac{\left | J^\eta_{i,\epsilon} \right |^2}{D_i \left(c^\eta_{i,\epsilon} + \delta_m\right)} =
	\begin{cases}
		\frac{\left | J^\eta_{i,\epsilon} \right |^2}{D_i c^\eta_{i,\epsilon}} &\text{for a.e. $(t,x) \in (0,T) \times \Oe \setminus A$,} \\
		0  &\text{for a.e.  $(t,x) \in A$.}
	\end{cases}
\end{flalign*}
Consequently, by \cite{Bha22}[Proposition 3.1], there exists a constant $C >0$ independent of $\epsilon$ and $\eta$ such that  
\begin{equation}\label{app_lem1_000}
\scalebox{1.4}{$\displaystyle \sum_{i=1}^{P}$} \bigintsss_{(0,T) \times \Oe \setminus A} \frac{\left | D_i \nabla c^\eta_{i,\epsilon} + D_i \eta \nabla \left(c^\eta_{i,\epsilon}\right)^p+ D_i z_i c_{i, \epsilon}^\eta \nabla \phi_\epsilon^\eta \right |^2}{D_i c^\eta_{i, \epsilon}} \,d(t,x) \leq 	C. 
\end{equation}
Indeed, it is shown in the proof of Proposition 3.1 in \cite{Bha22} that
\begin{equation}\label{app_lem1_detail}
\frac{d}{dt} V_\epsilon^\eta (t) + \scalebox{1.4}{$\displaystyle \sum_{i=1}^{P}$} \bigintsss_{\Oe}\frac{\left|J^\eta_{i,\epsilon}\right|^2}{D_i c^\eta_{i,\epsilon}} \,dx = 0 \ \ \text{for a.e. $t \in (0,T)$},
\end{equation}
where $V_\epsilon^\eta(t) \leq C \ \forall t \in [0,T]$ for some constant $C>0$ independent of $\eta$ and $\epsilon$, and, as discussed before, the second term in the LHS of (\ref{app_lem1_detail}) is zero whenever $c^\eta_{i,\epsilon}$ vanishes. Now, integrating (\ref{app_lem1_detail}) with respect to time, we obtain (\ref{app_lem1_000}).

Hence, 
\begin{align}\label{app_lem1_0}
&\scalebox{1.4}{$\displaystyle \sum_{i=1}^{P}$} 	\bigintsss_{(0,T) \times \Oe \setminus A} \frac{\left | \nabla c^\eta_{i, \epsilon}\right |^2}{c^\eta_{i, \epsilon}} + \frac{\left | \eta \nabla \left(c^\eta_{i,\epsilon}\right)^p+ z_i c_{i, \epsilon}^\eta \nabla \phi_\epsilon^\eta \right |^2}{c^\eta_{i, \epsilon}} + \frac{8}{p} \eta \left | \nabla \left(c^\eta_{i, \epsilon}\right)^\frac{p}{2}\right |^2 \nonumber\\
&+ 2 z_i \nabla c^\eta_{i, \epsilon} \nabla \phi^\eta_\epsilon \ \,d(t,x)  \leq  	C. 
\end{align}
Then 
\begin{align*}
&\scalebox{1.4}{$\displaystyle \sum_{i=1}^{P}$} 	\bigintsss_{(0,T) \times \Oe \setminus A} \frac{\left | \nabla c^\eta_{i, \epsilon}\right |^2}{c^\eta_{i, \epsilon}+1} + \frac{\left | \eta \nabla \left(c^\eta_{i,\epsilon}\right)^p+ z_i c_{i, \epsilon}^\eta \nabla \phi_\epsilon^\eta \right |^2}{c^\eta_{i, \epsilon}+1} + \frac{8}{p} \eta \left | \nabla \left(c^\eta_{i, \epsilon}\right)^\frac{p}{2}\right |^2 \\
&+ 2 z_i \nabla c^\eta_{i, \epsilon} \nabla \phi^\eta_\epsilon \ \,d(t,x)  \leq  	C. 
\end{align*} 
Since $\nabla c_{i, \epsilon} = 0$ a.e. in $A$, we get 
\begin{align*}
&\scalebox{1.4}{$\displaystyle \sum_{i=1}^{P}$} \bigintsss_{0}^{T} \bigintsss_{\Oe}\frac{\left | \nabla c^\eta_{i, \epsilon}\right |^2}{c^\eta_{i, \epsilon}+1} + \frac{\left | \eta \nabla \left(c^\eta_{i,\epsilon}\right)^p+ z_i c_{i, \epsilon}^\eta \nabla \phi_\epsilon^\eta \right |^2}{c^\eta_{i, \epsilon}+1} + \frac{8}{p} \eta \left | \nabla \left(c^\eta_{i, \epsilon}\right)^\frac{p}{2}\right |^2 \\
&+ 2 z_i \nabla c^\eta_{i, \epsilon} \nabla \phi^\eta_\epsilon \ \,dx \,dt \leq  	C. 
\end{align*} 
Again, note that $\sqrt{c^\eta_{i, \epsilon} + 1} \in L^2(0,T;H^1(\Oe))$ and 
\begin{equation*}
\nabla \sqrt{c^\eta_{i, \epsilon} + 1} = \frac{1}{2 \sqrt{c^\eta _{i, \epsilon} + 1}} \nabla c^\eta_{i, \epsilon}. 
\end{equation*}
Therefore, 
\begin{align}\label{app_lem1_1}
&\scalebox{1.4}{$\displaystyle \sum_{i=1}^{P}$}  \bigintsss_{0}^{T} \bigintsss_{\Oe} 4 \left | \nabla \sqrt{c^\eta_{i, \epsilon} + 1}\right|^2   + \frac{\left | \eta \nabla \left(c^\eta_{i,\epsilon}\right)^p+ z_i c_{i, \epsilon}^\eta \nabla \phi_\epsilon^\eta \right |^2}{c^\eta_{i, \epsilon}+1} + \frac{8}{p} \eta \left | \nabla \left(c^\eta_{i, \epsilon}\right)^\frac{p}{2}\right |^2 \nonumber\\
&+ 2 z_i \nabla c^\eta_{i, \epsilon} \nabla \phi^\eta_\epsilon  \ \,dx \,dt   \leq  	C. 
\end{align} 
The first three terms of each integral are non-negative. So, let us estimate $\displaystyle \sum_{i=1}^{P} \int_{0}^{T} \int_{\Oe}z_i  \nabla c^\eta_{i, \epsilon} \nabla \phi^\eta_\epsilon \, dx \,dt$. Integrating by parts, we have 
\begin{align}\label{app_lem1_1a}
&\int_{0}^{T} \int_{\Oe} \sum_{i=1}^{P}  z_i \nabla c^\eta_{i, \epsilon} \nabla \phi^\eta_\epsilon  \ \,dx \,dt =  \int_{0}^{T} \int_{\Oe} \left | \Delta \phi_\epsilon^\eta \right | ^2 \,dx \,dt \nonumber\\
&+ \sum_{i=1}^{P}  \int_{0}^{T} \int_{\partial \Oe} z_i c^\eta_{i, \epsilon} \xi_\epsilon \,dS(x) \,dt . 
\end{align}
Therefore, it remains to have an estimate of $\displaystyle \sum_{i=1}^{P}  \int_{0}^{T} \int_{\partial \Oe} z_i c^\eta_{i, \epsilon} \xi_\epsilon \,dS(x) \,dt$. Now, 
\begin{eqnarray}\label{app_lem1_2}
  \int_{0}^{T} \int_{\partial \Oe} z_i c^\eta_{i, \epsilon} \xi_\epsilon \,dS(x) \,dt \geq -  \int_{0}^{T} \int_{\partial \Oe} |z_i| ( c^\eta_{i, \epsilon} +1 ) | \xi_\epsilon |\,dS(x) \,dt \nonumber\\
  \geq -C \epsilon   \int_{0}^{T} \int_{\Gamma_\epsilon} \left | \sqrt{c^\eta_{i, \epsilon} +1 } \right |^2  \,dS(x) \,dt - C  \int_{0}^{T} \int_{\partial \Omega}  \left | \sqrt{c^\eta_{i, \epsilon} +1 } \right |^2 \,dS(x) \,dt \nonumber \\
\end{eqnarray}
By Lemma \ref{TraceInequalityLemma}, we have 
\begin{align}\label{app_lem1_3}
&\epsilon   \int_{0}^{T} \int_{\Gamma_\epsilon} \left | \sqrt{c^\eta_{i, \epsilon} +1 } \right |^2  \,dS(x) \,dt \nonumber\\
&\leq  C  \left (  \int_{0}^{T} \int_{\Oe}  \left | \sqrt{c^\eta_{i, \epsilon} +1 } \right |^2  \,dx \,dt + \epsilon ^2  \int_{0}^{T} \int_{\Oe}  \left | \nabla \sqrt{c^\eta_{i, \epsilon} +1 } \right |^2  \,dx \,dt  \right ). 
\end{align}
Also, using the weighted trace inequality \cite[p. 63, Exercise II.4.1]{Gal11} and extension operator (Lemma \ref{lemma_extension}), we get for any $\delta >0$: 
\begin{eqnarray}\label{app_lem1_4}
	  \int_{0}^{T} \int_{\partial \Omega}  \left | \sqrt{c^\eta_{i, \epsilon} +1 } \right |^2 \,dS(x) \,dt =  	  \int_{0}^{T} \int_{\partial \Omega}  \left | \widetilde{\text{$\sqrt{c^\eta_{i, \epsilon} +1 }$}} \right |^2 \,dS(x) \,dt \nonumber \\
	  \leq \delta   \int_{0}^{T} \int_{\Omega}  \left | \nabla  \widetilde{\text{$\sqrt{c^\eta_{i, \epsilon} +1 }$}} \right |^2 \,dx \,dt  + C (\delta) \int_{0}^{T} \int_{\Omega}  \left | \widetilde{\text{$\sqrt{c^\eta_{i, \epsilon} +1 }$}} \right |^2 \,dx \,dt  \nonumber \\
	  \leq C \delta   \int_{0}^{T} \int_{\Omega_\epsilon}  \left | \nabla \sqrt{c^\eta_{i, \epsilon} +1 } \right |^2 \,dx \,dt  + C (\delta) \int_{0}^{T} \int_{\Omega_\epsilon}  \left |  \sqrt{c^\eta_{i, \epsilon} +1 }\right |^2 \,dx \,dt . 
\end{eqnarray}
Hence, using (\ref{app_lem1_1a}), (\ref{app_lem1_2}),  (\ref{app_lem1_3}),  (\ref{app_lem1_4}) in  (\ref{app_lem1_1}), we obtain
\begin{align}\label{app_lem1_5}
	&\scalebox{1.4}{$\displaystyle \sum_{i=1}^{P}$} \bigintsss_{0}^{T} \bigintsss_{\Oe} 4 \left | \nabla \sqrt{c^\eta_{i, \epsilon} + 1}\right|^2   + \frac{\left | \eta \nabla \left(c^\eta_{i,\epsilon}\right)^p+ z_i c_{i, \epsilon}^\eta \nabla \phi_\epsilon^\eta \right |^2}{c^\eta_{i, \epsilon}+1} + \frac{8}{p} \eta \left | \nabla \left(c^\eta_{i, \epsilon}\right)^\frac{p}{2}\right |^2 + \left | \Delta \phi_\epsilon^\eta \right | ^2  dx \,dt \nonumber \\
	   & \leq	C    \scalebox{1.2}{$\displaystyle \sum_{i=1}^{P}$}  \int_{0}^{T} \int_{\Oe}  \left | \sqrt{c^\eta_{i, \epsilon} +1 } \right |^2  + \epsilon ^2   \left | \nabla \sqrt{c^\eta_{i, \epsilon} +1 } \right |^2 + \delta    \left | \nabla \sqrt{c^\eta_{i, \epsilon} +1 } \right |^2  + C (\delta)   \left |  \sqrt{c^\eta_{i, \epsilon} +1 }\right |^2 \,dx \,dt. \nonumber \\
\end{align} 
We choose $\delta, \epsilon$ small enough so that $C(\epsilon ^2 + \delta)   \left | \nabla \sqrt{c^\eta_{i, \epsilon} +1 } \right |^2$ in the RHS of (\ref{app_lem1_5}) can be absorbed by $ 4 \left | \nabla \sqrt{c^\eta_{i, \epsilon} +1 } \right |^2$ in the LHS. Finally, using the conservation of mass (Lemma \ref{lemma_conv_mass}), our proof is complete. 
\end{proof}
\begin{proof}[Proof of Lemma \ref{conv_approx_micro}] Recall that $V_\epsilon := H^1 (\Oe) \cap L^\infty (\Oe)$ is a Banach space equipped with the norm $\lVert . \rVert_{V_\epsilon} : = \max \{\lVert . \rVert_{H^1(\Oe)} ,\lVert . \rVert_{ L ^\infty(\Oe)} \}$. Observe the embeddings: 
	\begin{equation}\label{appen_banach_triple}
	H^1(\Oe) \subset \subset L^2(\Oe) \subset V_\epsilon ^\prime. 
	\end{equation}
Next, we show strong convergence, up to a subsequence, of $c^\eta_{i,\epsilon}$ as $\eta \to 0$, using the above triple of Banach spaces. We note that the space $V_\epsilon ^\prime$ also played an important role in extracting a strongly convergent subsequence of $\widetilde{c_{i,\epsilon}}$ as $\epsilon \to 0$ (see Proposition \ref{th_st_cut_off} and Theorem \ref{thm_st_conv_micro_1}).  

A simple density argument yields that $\partial_{t} \sqrt{c^\eta_{i, \epsilon}+1} \in L^1(0,T;V_\epsilon^\prime)$ with
\begin{equation*}
\left \langle \partial_{t}  \sqrt{c^\eta_{i, \epsilon}+1} (t) , \upsilon \right  \rangle _{V_\epsilon^\prime, V_\epsilon}= \left \langle \partial_{t}  c^\eta_{i, \epsilon} (t), \frac{\upsilon}{2 \sqrt{c^\eta_{i, \epsilon}+1}} \right \rangle_{H^1(\Oe)^\prime, H^1(\Oe)} \ \forall \upsilon \in V_\epsilon, \ \text{for a.e. $t \in (0,T)$.} 
\end{equation*} 
Next, we claim that $\partial_{t} \sqrt{c^\eta_{i, \epsilon}+1}$ is bounded in $L^1(0,T;V_\epsilon^\prime)$ uniformly in $\eta$. For all $\upsilon \in V_\epsilon$, we have:
\begin{eqnarray*}
 \left \langle \partial_{t}  c^\eta_{i, \epsilon} (t), \frac{\upsilon}{2 \sqrt{c^\eta_{i, \epsilon}+1}} \right \rangle_{H^1(\Oe)^\prime, H^1(\Oe)} = \bigintsss_{\Oe} J^\eta_{i, \epsilon} \cdot \nabla \left(\frac{\upsilon}{2 \sqrt{c^\eta_{i, \epsilon}+1}}\right) \,dx  \\
= \bigintsss_{\Oe} J^\eta_{i, \epsilon} \cdot \left [\frac{1}{2 \sqrt{c^\eta_{i, \epsilon}+1}} \nabla \upsilon - \frac{ \upsilon}{4 \left (c^\eta_{i, \epsilon}+1 \right)^{\frac{3}{2}}} \nabla c^\eta_{i, \epsilon}\right ] \,dx \\
\leq \left \lVert \frac{J^\eta_{i,\epsilon}}{2 \sqrt{c^\eta_{i, \epsilon}+1}}\right \rVert _{L^2(\Oe)}  \left \lVert \nabla \upsilon \right \rVert _{L^2(\Oe)} +  \left \lVert \frac{J^\eta_{i,\epsilon}}{2 \sqrt{c^\eta_{i, \epsilon}+1}}\right \rVert _{L^2(\Oe)} \left \lVert \frac{\nabla c^\eta_{i,\epsilon}}{2 \sqrt{c^\eta_{i, \epsilon}+1}}\right \rVert _{L^2(\Oe)}  \left \lVert \upsilon \right \rVert _{L^\infty(\Oe)}. 
\end{eqnarray*}
Consequently,  Lemma \ref{uni_est_prop_st} and (\ref{app_lem1_000}) prove our claim. Again, by conservation of mass (Lemma \ref{lemma_conv_mass}) and Lemma \ref{uni_est_prop_st}, we have $\sqrt{c^\eta_{i, \epsilon}+1}$ is bounded in $L^2(0,T;H^1(\Oe))$. . Hence, by (\ref{appen_banach_triple}), we can apply \cite[Corollary 4]{Sim87} to conclude that, up to a subsequence, $\sqrt{c^\eta_{i, \epsilon}+1}$ strongly converges in $L^2(0,T;L^2(\Oe))$ as $\eta \to 0$. 

Now that we have the strong convergence, we continue our arguments as in \cite[Theorem 1]{Both14}. For example, we note that $\nabla c^\eta_{i, \epsilon} = 2 \sqrt{c^\eta_{i,\epsilon}+1}  \ \nabla \sqrt{c^\eta_{i,\epsilon}+1}$. Now, the weak convergence of $\nabla c^\eta_{i, \epsilon}$ to $\nabla c_{i, \epsilon}$ in $L^1((0,T) \times \Oe)$ follows from observing the facts that $\sqrt{c^\eta_{i, \epsilon}+1}$ converges strongly in $L^2((0,T) \times \Oe)$ and $\nabla \sqrt{c^\eta_{i,\epsilon}+1}$ is weakly relatively compact in $L^2((0,T) \times \Oe)$ (Lemma \ref{uni_est_prop_st}). 

Finally, we conclude by commenting on the convergence (up to a subsequence) of $J^\eta_{i, \epsilon} + D_i \nabla c^\eta_{i, \epsilon}$, as $\eta \to 0$, to $-D_i z_i c_{i, \epsilon} \nabla \phi_\epsilon$ weakly in $L^1((0,T) \times \Oe)$. Let $A$ be the set as in the proof of Lemma \ref{uni_est_prop_st}; i.e., $A := \{(t,x)\in (0,T) \times \Omega_{\epsilon}: c^\eta_{i,\epsilon} (t,x) = 0\}.$ We write
\begin{equation*}
J^\eta_{i, \epsilon} + D_i \nabla c^\eta_{i, \epsilon} = - D_i \eta \nabla \left( c^\eta_{i, \epsilon} \right) ^p - D_i z_i c^\eta_{i, \epsilon} \nabla \phi_\epsilon^\eta =  -D_i \sqrt{c^\eta_{i, \epsilon}} \  \mathcal{F}(c^\eta_{i, \epsilon}, \nabla c^\eta_{i, \epsilon}, \nabla \phi_\epsilon^\eta),
\end{equation*}
where  $\mathscr{F}(c^\eta_{i, \epsilon}, \nabla c^\eta_{i, \epsilon}, \nabla \phi_\epsilon^\eta) =  \mathscr{F}_1( c^\eta_{i, \epsilon}, \nabla c^\eta_{i, \epsilon})  +  \mathscr{F}_2(c^\eta_{i, \epsilon}, \nabla \phi_\epsilon^\eta)$ with 
	\begin{flalign*}
\mathscr{F}_1( c^\eta_{i, \epsilon}, \nabla c^\eta_{i, \epsilon}) =	 
	\begin{cases}
	 \frac{\eta \nabla \left( c^\eta_{i, \epsilon} \right) ^p  }{\sqrt{c^\eta_{i, \epsilon}}} &\text{for a.e. $(t,x) \in (0,T) \times \Oe \setminus A$,} \\  
		0  &\text{for a.e.  $(t,x) \in A$;}
	\end{cases}
\end{flalign*}
and 
	\begin{flalign*}
	\mathscr{F}_2(c^\eta_{i, \epsilon}, \ \nabla \phi_\epsilon^\eta) =	 
	\begin{cases}
		\frac{z_i c^\eta_{i, \epsilon} \nabla \phi_\epsilon^\eta}{\sqrt{c^\eta_{i, \epsilon}}} &\text{for a.e. $(t,x) \in (0,T) \times \Oe \setminus A$,} \\  
		0  &\text{for a.e.  $(t,x) \in A$.}
	\end{cases}
\end{flalign*}
We claim that $\mathscr{F}(c^\eta_{i, \epsilon}, \nabla c^\eta_{i, \epsilon}, \nabla \phi_\epsilon^\eta)$ is weakly relatively compact in $L^2((0,T)  \times \Oe)$. To see this, note that, in the proof of Lemma \ref{uni_est_prop_st}, $\sqrt{c^\eta_{i, \epsilon}+1}$ can be replaced by $\sqrt{c^\eta_{i, \epsilon}+ \rho_m}$, where $ \rho_m$ is a strictly decreasing sequence of positive numbers such that $ \rho_m \to 0$ as $m \to \infty$. So, as in (\ref{app_lem1_5}), we get 
\begin{eqnarray*}
	\scalebox{1.4}{$\displaystyle \sum_{i=1}^{P}$} \bigintsss_{0}^{T} \bigintsss_{\Oe}    \frac{\left | \eta \nabla \left(c^\eta_{i,\epsilon}\right)^p+ z_i c_{i, \epsilon}^\eta \nabla \phi_\epsilon^\eta \right |^2}{c^\eta_{i, \epsilon}+\rho_m} \,dx \,dt 
	\leq	C   . 
\end{eqnarray*} 
Consequently, by the monotone convergence theorem, we have that  $\mathscr{F}(c^\eta_{i, \epsilon}, \nabla c^\eta_{i, \epsilon}, \nabla \phi_\epsilon^\eta)$ is bounded in $L^2((0,T)  \times \Oe)$, which proves our claim. 

 Again, one has $\displaystyle \mathscr{F}_1(c^\eta_{i, \epsilon}, \nabla c^\eta_{i, \epsilon})= \frac{\eta p}{p-\frac{1}{2}} \nabla \left( c^\eta_{i, \epsilon}\right) ^{p-\frac{1}{2}}$ in $(0,T) \times \Oe$. Then following the same arguments from \cite[Theorem 1]{Both14}, we are done. 
\end{proof}
\end{appendix}
\section*{Acknowledgments}
The author is thankful to Prof. Dr. Maria Neuss-Radu for many fruitful discussions. Furthermore, the author would like to thank the anonymous reviewers for their constructive comments and suggestions, which improved the quality of the manuscript.
\section*{Funding}
Parts of this work were supported by the Kempe Foundations (JCSMK22-0153), Sweden, and the Chair of Applied Mathematics (Modeling and Numerics), FAU Erlangen-Nürnberg, Germany.

\end{document}